\documentclass[mlq]{w-art}
\usepackage{amsthm, amsmath, amssymb, amscd,  amsfonts, amstext,
  amsbsy,enumerate,mathrsfs}
\usepackage[all]{xy}

\newcommand{\pow}{\ensuremath{\mathscr{P}}}
\def\leng{{\rm lh}}
\newcommand{\Bai}{\ensuremath{{}^\omega \omega}}
\newcommand{\seqo}{\ensuremath{{}^{<\omega}\omega}}
\newcommand{\AD}{\ensuremath{{\rm \mathsf{AD}}}}
\newcommand{\ZF}{\ensuremath{{\rm \mathsf{ZF}}}}
\newcommand{\DCR}{\ensuremath{{\rm \mathsf{DC}(\mathbb{R})}}}
\newcommand{\ACOR}{\ensuremath{{\rm \mathsf{AC}_\omega(\mathbb{R})}}}
\newcommand{\BP}{\ensuremath{{\rm \mathsf{BP}}}}
\newcommand{\AC}{\ensuremath{{\rm \mathsf{AC}}}}
\newcommand{\ADL}{\ensuremath{{\rm \mathsf{AD^L}}}}
\newcommand{\ADW}{\ensuremath{{\rm \mathsf{AD^W}}}}
\newcommand{\SLOW}{\ensuremath{{\rm \mathsf{SLO^W}}}}
\newcommand{\SLO}{\ensuremath{{\rm \mathsf{SLO}}}}
\newcommand{\bt}{\ensuremath{{\rm \mathsf{bt}}}}
\newcommand{\p}{\ensuremath{{\rm \mathsf{p}}}}
\newcommand{\pI}{\ensuremath{{\rm \mathbf{I}}}}
\newcommand{\pII}{\ensuremath{{\rm \mathbf{II}}}}
\newcommand{\conc}{{}^\smallfrown}
\newcommand{\id}{{\rm id}}
\newcommand{\ax}[1]{\ensuremath{{\rm \mathsf{#1}}}}

\newcommand{\imp}{\Rightarrow}

\newcommand{\bSigma}{\mathbf{\Sigma}}
\newcommand{\bPi}{\mathbf{\Pi}}
\newcommand{\bGamma}{\ensuremath{\mathbf{\Gamma}}}
\newcommand{\bDelta}{\mathbf{\Delta}}
\newcommand{\bN}{\mathbf{N}}
\newcommand{\F}{\mathcal{F}}
\newcommand{\G}{\mathcal{G}}
\newcommand{\B}{\mathcal{B}}
\newcommand{\Lip}{\ensuremath{{\mathsf{Lip}}}}
\renewcommand{\L}{\ensuremath{{\mathsf{L}}}}
\newcommand{\W}{\ensuremath{{\mathsf{W}}}}
\newcommand{\restr}[2]{#1 \restriction #2}
\renewcommand{\sf}[1]{\ensuremath{{\mathsf{#1}}}}
\newcommand{\seq}[2]{\langle #1 \mid  #2 \rangle}
\newcommand{\onto}{\twoheadrightarrow}

\newtheorem{theorem}{Theorem}[section]
\newtheorem{lemma}[theorem]{Lemma}
\newtheorem{corollary}[theorem]{Corollary}
\newtheorem{proposition}[theorem]{Proposition}
\newtheorem{question}{Question}
\newtheorem{claim}{Claim}[theorem]
\theoremstyle{definition}
\newtheorem{defin}{Definition}
\theoremstyle{remark}
\newtheorem{remark}[theorem]{Remark}
\newtheorem{example}[theorem]{Example}

\begin{document}

\title{Game representations of classes of piecewise definable functions}
\author[L. Motto Ros]{Luca Motto Ros
\footnote{E-mail:~\textsf{luca.mottoros@libero.it},
          Phone: +43\,1\,4277\,50508,
          Fax: +43\,1\,4277\,50599}}
\date{\today}
\address{Kurt G\"odel Research Center for Mathematical Logic,
  University of Vienna,  W\"ahringer Stra{\ss}e 25,  
 A-1090 Vienna, 
Austria}

\begin{abstract}
  We  present a general way of defining various \emph{reduction games} 
on $\omega$ which ``represent'' corresponding topologically defined
classes of functions. 
In
particular,  we will show how to construct games for
piecewise defined functions, for functions which are pointwise limit of
certain sequences of functions and for $\bGamma$-measurable
functions. These games turn out to be useful as a combinatorial tool
for the study of general reducibilities for subsets of the Baire space 
\cite{mottorosborelamenability}.
\end{abstract}

\maketitle

\section{Introduction}

The first reduction games which have appeared in the literature are perhaps
the Lipschitz game $G_\L$ and the Wadge game $G_\W$ (both were defined
by Wadge in his Ph.D.\ Thesis, see \cite{wadgethesis}). They are a
special kind  of
infinite two-player zero sum games 
on $\omega$ with perfect information, and are designed in such a way
that if player $\pII$ has a 
\emph{legal} strategy $\tau$  
in $G_\L$  (resp.\ $G_\W$) then from $\tau$ it can be recovered 
in a canonical and fixed way
 a function $f_\tau$ from the Baire space $\Bai$ into itself which is 
Lipschitz with constant $1$ (resp.\ continuous). Conversely, given 
a Lipschitz with constant $1$ (resp.\ continuous) function $f \colon  \Bai \to \Bai$ one
can construct 
a legal strategy $\tau$ for player $\pII$ in the corresponding game such that
$f = f_\tau$.
These games were introduced to study the relations (also called
\emph{reducibilities})
$\leq_\L$ and $\leq_\W$, where for every $A,B \subseteq \Bai$
\begin{align*}
 A \leq_\L B \text{ (resp.\ }A \leq_\W B \text{)}\iff A = f^{-1}(B) 
& \text{ for some Lipschitz with constant } 1 \\
&\text{ (resp.\ continuous) function }f.
 \end{align*}
The link between these preorders and the corresponding games is the following:
given  
$A, B \subseteq \Bai$, a payoff set for $G_\L$ (resp.\ $G_\W$) is
canonically constructed 
(see Section \ref{sectionreductiongames}) in such a way that player
$\pII$ has a \emph{winning} 
strategy in $G_\L(A,B)$ 
(resp.\ $G_\W(A,B)$) if and only if $A \leq_\L B$ (resp.\ $A \leq_\W B$).
Assuming the \emph{Axiom of Determinacy} $\AD$, or even just the determinacy of
the corresponding games $G_\L$ and $G_\W$, 
Wadge proved that both $\leq_\L$ and $\leq_\W$ induce 
well-behaved stratifications of the subsets
of $\Bai$ which have turned out to be very useful in
various parts of 
Set Theory (see e.g.\ \cite{andrettaslo, mottorosborelamenability}).

Some years later, Van Wesep defined, building on work of Wadge,
another reduction  
game, 
the \emph{backtrack game} $G_\bt$, but at that time it was not clear
which should be 
the ``topological'' class of functions $\F$ corresponding 
to  legal strategies for player $\pII$ in 
$G_\bt$. 
It was Andretta who solved this problem in \cite{andrettamoreonwadge},
by showing that such $\F$ 
is exactly  
the collection of those $f\colon  \Bai \to \Bai$ for which there is a
partition $\seq{P_n}{n \in \omega}$ 
of $\Bai$ into closed sets such that $\restr{f}{P_n}$ is continuous
for every $n \in \omega$ (see \cite[Theorem 21]{andrettamoreonwadge}),
which in turn coincide with the  collection 
of the $\bDelta^0_2$-functions by a theorem of Jayne and Rogers
 (see e.g.\ \cite[Theorem 1.1]{mottorossemmes} for
a proof of this last result). 
Another reduction game, namely the \emph{eraser
  game} $G_\sf{E}$, was defined 
(essentially) by Duparc  in such a way that $f \colon \Bai \to \Bai$ is a
Baire class $1$ function  
if and only if there is some legal strategy $\tau$ for $\pII$ in 
$G_{\sf{E}}$ such that $f = f_\tau$. Finally, some work related to
this topic 
was developed in \cite{kechrisdeterminacy} (although in this case
there are no reduction games 
directly involved).

Having all these useful reduction games,
it is quite natural to ask if one could
also define  reduction games for other ``natural'' collections of
functions (this question 
was explicitly posed by Andretta in his \cite{andrettaslo}: ``Is there a
Wadge-style game for higher 
levels of reducibility, like $\bDelta^0_3$ and such?''). 
More precisely: say that a set of functions $\F$  is
\emph{playable} if there is  
some reduction game $G_*$ such that for every $f \colon \Bai \to
\Bai$, $f \in \F$ if and only if 
there is a legal strategy $\tau$ for $\pII$ in $G_*$ for which $f =
f_\tau$ (this notion 
will be completely formalized in  Section
\ref{sectionreductiongames}). Clearly not every set of
 functions is playable: for example, the collection of \emph{all}
 functions from $\Bai$ into 
 itself is not playable, as a simple cardinality argument shows (the
 strategies for
\emph{any} game on $\omega$ are always at most
$2^{\aleph_0}$). Nevertheless we can ask the following:

\begin{question}\label{question}
Which (topologically defined) classes of functions are playable? 
\end{question}

The motivation for this
problem mainly relies on the fact that the presence of a reduction game
provides combinatorial tools for the study of the reducibility induced
by the corresponding set 
of functions --- see e.g.\ \cite{wadgethesis, andrettamoreonwadge}.

The first partial answer to this general problem was
given by Semmes in \cite{semmesmultitapegames} and in his Ph.D.\ thesis 
\cite{semmesthesis}: there he proposed a game (called \emph{tree
  game}) which corresponds 
to the Borel functions, and some other games (the \emph{multitape
  game} $G_\sf{M}$, the \emph{multitape eraser} game $G_{\sf{ME}}$,
and the game $G_{1,3}(f)$) which correspond, respectively, 
to the functions strictly continuous on a
$\bPi^0_2$-partition, to the
functions which are of Baire class $1$ on a $\bPi^0_2$-partition, and
to the Baire class $2$ functions ---
see next section for the definitions of these classes of
functions. 

In this paper we somewhat extend these results providing a positive answer to Question
\ref{question} for  a wide class of subsets of the Borel functions, and
for $\bGamma$-measurable functions (where $\bGamma$ is any boldface
pointclass closed under countable unions and finite intersections):
therefore the paper is in some respect unusual for a research
publication in mathematics, as it mainly consists of definitions and
of proofs that these definitions are correct. Nevertheless, in the
last section we will also provide some applications of these games
which motivate our interest in this subject.  

The material of this paper (except for Section \ref{sectionGamma})
mainly comes from Sections 2.2. and 4.8 of the author's Ph.D.\ 
thesis 
\cite{mottorosthesis} or is obtained via minor variations of the
constructions contained therein, but for the reader's convenience (and to avoid
  confusions) in the present paper we have adapted most of the
  terminology and notation  used in \cite{mottorosthesis} to the one
  already used in \cite{semmesmultitapegames}, with the following
  exception: because of the 
  applications of reduction games to reducibilities for sets of reals
  given in Section \ref{sectionaxioms},  in this paper the payoff set
  of a reduction game  will
  be defined starting from two sets of reals (see Section
  \ref{sectionreductiongames}), whereas 
   in
  \cite{semmesmultitapegames} it was defined starting from a (partial)
  function from the reals into the reals (nevertheless, it is quite easy
  to see how to modify one kind of presentation into the other).

The constructions we are going to present rely on a very general way
of defining games for sets of 
functions which are piecewise defined,
for sets of functions which are (pointwise) limits of certain
sequences of functions, and for $\bGamma$-measurable functions: most
of the proofs involve some sort of operation for games which allows to
transform a sequence of already known reduction games (representing
some classes of functions) into a new reduction game which represent
the larger class of those functions piecewise in the old classes on a
definable partition, or the class of the pointwise limits of the old
functions.  

The paper is reasonably self-contained and is organized as follows: in
Section \ref{sectionpreliminaries} 
we will fix some notation, while in Section \ref{sectionreductiongames}
we will give
a precise definition of what should be meant by \emph{reduction game} and
\emph{playable} set of functions, and give some 
basic examples (both old and new). In Sections \ref{sectionplayable} and \ref{sectionGamma}
we  will prove our main results,
showing how to construct new reduction games (building
on other known games): this will give a ``uniform'' solution to our problem 
for almost all sets of functions involved in (generalizations of) Wadge's theory, for
$\bGamma$-measurable functions, and
for some other related sets of functions. In Section
\ref{sectionaxioms} we will give some examples of
how  to apply
the techniques arising from these games to the study of various 
reducibilities, 
and we will prove some relationships between the corresponding
determinacy axioms. Finally, in
Section \ref{sectionsmallness} we will give the optimal condition
under which constructions like those presented in Section
\ref{sectionplayable} can be 
carried out: even if this technical improvement allows to deal
with a strictly larger class of sets of functions, for the sake of simplicity we
have postponed it to the 
last section  because it complicates very much the
presentation without adding relevant ideas for the construction of the new
reduction games.

\section{Preliminaries and notation}\label{sectionpreliminaries}

In most of the applications involving games, one usually assumes $\AD$
(or some other axiom of this kind) and then uses the combinatorics
arising from the winning strategies in the games under consideration
to prove the desired results. However, $\AD$ (and, in general, all
known determinacy principles which are not restricted to the context
of a small definable pointclass, like the pointclass of Borel sets)
contradicts the full axiom of choice $\AC$, and 
therefore in presenting new games and their applications
one has to be careful and just use choice principles 
which do not contradict $\AD$. In this paper,  
we will always work in $\ZF+\ACOR$ except for Section
\ref{sectionaxioms}, in which we will sometimes need the \emph{Axiom
  of Dependent Choice (over the reals)} $\DCR$ and the axiom $\BP$,
that is the statement ``every set of reals has the Baire property''. 

Our notation is quite standard and
we refer the reader to the monograph \cite{kechris} for all the
undefined symbols and notions. 
 Given two sets $A$ and $B$, we will
denote by ${}^B A$ 
the collection of all functions from $B$ to $A$. Thus we will denote by
${}^\omega A$
the set of all $\omega$-sequences of elements of $A$, while the
collection of the \emph{finite} sequences of elements of $A$ will be denoted
by ${}^{<\omega}A$ (we will refer to the length of a finite sequence $s$
 with the symbol $\leng(s)$). 
As usual in Descriptive Set Theory,  the elements
of the Baire 
space $\Bai$ will be
called \emph{reals}. 
If $n,k \in \omega$ we will write $
n^{(k)}$ for the sequence $\langle
\underbrace{n,\dotsc,n}_k \rangle$ and $\vec{n}$ for
the $\omega$-sequence  
$\langle n,n,n, \dotsc \rangle$. 
For simplicity of notation,
we will also put $\bSigma^0_{<\xi} = \bigcup_{\mu < \xi} \bSigma^0_\xi$,
$\bPi^0_{<\xi} = \bigcup_{\mu<\xi} \bPi^0_\mu$ and 
$\bDelta^0_{<\xi} = \bigcup_{\mu<\xi} \bDelta^0_\mu$. Let $\langle \cdot,\cdot \rangle\colon
\omega \times \omega \to \omega$ be 
the  bijection
$\langle n,m \rangle =2^n(2m+1)-1$. Then we can define the homeomorphism
 (where ${}^\omega A$ is endowed with
the product topology of the discrete topology on $A$)
\[ \bigotimes \colon  {}^\omega ({}^\omega A) \to {}^\omega A \colon
\langle x_n\mid  n \in 
\omega \rangle \mapsto \bigotimes\nolimits_n x_n =
\seq{x_n(m)}{\langle n,m \rangle \in \omega} ,\]
and, conversely, the ``projections''
$\pi_n\colon  {}^\omega A \to {}^\omega A$ defined by $\pi_n(x) = \langle
x(\langle n,m \rangle)\mid m \in \omega \rangle$: clearly, every $\pi_n$
 is
surjective, continuous and open. 

Unless otherwise
specified, in what follows all  functions should be intended  
as \emph{partial} functions, 
i.e.\ defined just on some $X \subseteq \Bai$ (endowed with the
relative topology inherited from $\Bai$), and not 
necessarily on the whole space $\Bai$.
We denote by
 $\Lip(2^k)$ the collection of the Lipschitz functions with
constant less or equal than $2^k$, and put $\Lip =
\bigcup_{k \in \omega} \Lip(2^k)$.
Since they played a special role as reducibilities, we will denote by $\L$ the set $\Lip(1)$,
and by  $\W$ the set of all
continuous functions.
Moreover, the collection of the
Baire class $\xi$ functions (equivalently, $\bSigma^0_{\xi+1}$-measurable functions --- see the next paragraph for the definition)
will be denoted by $\B_\xi$.
Finally, given any nonzero countable ordinal $\xi$, we will denote
by $\sf{D}_\xi$ the collection of all \emph{$\bDelta^0_\xi$-functions}, i.e.\ 
of those $f \colon X \to \Bai$ such that $f^{-1}(D) \in
\bDelta^0_\xi(X)$ for every $D \in \bDelta^0_\xi$. 

A pointclass $\bGamma \subseteq \pow(\Bai)$ is said to be \emph{boldface} if it is closed
under continuous preimages, and is called \emph{$\bSigma$-pointclass} if
it is boldface and closed under countable unions and finite
intersections. A set $S 
\subseteq \Bai \times \Bai$ is said to be \emph{universal for
  $\bGamma$} if $\bGamma(\Bai) = \{ S_x \mid x \in \Bai\}$, where $S_x
= \{ y \in \Bai \mid (x,y) \in S\}$. A function
$f \colon X \to \Bai$ is \emph{$\bGamma$-measurable} if $f^{-1}(U) \in
\bGamma(X)$ for every open set $U$. The collection of such functions is
denoted by $\F_\bGamma$. Note that if $\bGamma$ is a
$\bSigma$-pointclass then $f \in \F_\bGamma$ if and only if
$f^{-1}(U_n) \in \bDelta_\bGamma(X)
= \bGamma(X) \cap \breve{\bGamma}(X)$ for any clopen subbasis $\{ U_n \mid n \in
\omega\}$ of the topology of $\Bai$. 
A \emph{$\bGamma$-partition} of $X \subseteq \Bai$ is a family
 $\langle D_n \mid n \in \omega \rangle$ of pairwise disjoint sets
 of $\bGamma(X)$ such that $X = \bigcup_{n \in \omega} D_n$. If
 $\bGamma$ is a $\bSigma$-pointclass then every $\bGamma$-partition
 of $X$ is automatically a $\bDelta_\bGamma$-partition. Given a
 sequence\footnote{When $\F_n = \F$ 
for every $n$ we will systematically use the symbol $\F$ instead of
$\vec{\F}$  in all the notation.}  $\vec{\F}  
=\F_0, \F_1, \dotsc$
of sets
of functions and a $\bSigma$-pointclass $\bGamma$, we will denote by
$\sf{D}_\bGamma^{\vec{\F}}$  the collection
of those $f 
\colon X \to \Bai$ for which there is a $\bGamma$-partition
(equivalently,
a $\bDelta_\bGamma$-partition) 
$\seq{D_n}{n \in \omega}$ of $X$ and
a sequence $f_0, f_1, \dotsc$ of functions \emph{with domain $X$} such
that $f_n \in \F_n$ and  
$\restr{f}{D_n} = \restr{f_n}{D_n}$ for every $n \in \omega$. 
If $\bGamma = \bSigma^0_\xi$ for some $\xi$, we will simply write
$\sf{D}^{\vec{\F}}_\xi$ instead of $\sf{D}^{\vec{\F}}_{\bSigma^0_\xi}$ (in this case we
can replace ``$\bDelta^0_\xi$-partition'' with
``$\bPi^0_{<\xi}$-partition'' in the definition above).
In particular, a function in $\sf{D}^\W_\xi$ will be said
\emph{continuous on a $\bDelta^0_\xi$-partition} 
(equiv.\ \emph{on a $\bPi^0_{<\xi}$-partition}).
The following minor variation of the previous definition
in general gives a different 
set of functions --- see \cite[Remark 6.2]{mottorosborelamenability}:
$\tilde{\sf{D}}_\bGamma^{\vec{\F}}$ (resp.\
$\tilde{\sf{D}}_\xi^{\vec{\F}}$) denotes 
 the collection of those $f \colon X \to \Bai$ for which there is a
 $\bDelta_\bGamma$-partition   (resp.\ a $\bDelta^0_\xi$-partition or,
 equivalently, a $\bPi^0_{<\xi}$-partition) 
$\seq{D_n}{n \in \omega}$ of $X$  such that 
$\restr{f}{D_n} \in \F_n$ for every $n \in \omega$.
A function in
$\tilde{\sf{D}}_\xi^\W$ will be said 
\emph{strictly continuous on a $\bDelta^0_\xi$-partition} (equiv.\ \emph{on
  a $\bPi^0_{<\xi}$-partition}). 
Both the previous definitions are useful: for instance $\sf{D}^\W_\xi$
has been  
used in \cite[Section 6]{mottorosborelamenability} as a natural example of
Borel-amenable reducibility, 
while $\tilde{\sf{D}}^\W_3$ has been used in \cite[Section 5.2]{semmesthesis}
to find a generalization for the level $3$ of the theorem of Jayne and
Rogers mentioned in the introduction, i.e.\ to show that $\sf{D}_3  =
\tilde{\sf{D}}^\W_3$.
Finally, given $\vec{\F}$ as above, we will denote by $\lim
\vec{\F}$ the collection of those 
$f \colon X \to \Bai$ for which there is a sequence of functions $f_0,
f_1, \dotsc$ with domain $X$ such that 
$f_n \in \F_n$ and $f$ is the \emph{pointwise} limit of the sequence 
$\seq{f_n}{n \in \omega}$.

In Section \ref{sectionaxioms}, we will deal with various
reducibilities for sets of reals and with the corresponding hierarchies of
complexity of $\pow(\Bai)$. Therefore for all the terminology and the
results about these concepts we refer the reader to
\cite{mottorosborelamenability} --- in fact we suggest to keep a copy of that
paper while reading that section in order to compare the various results with
the combinatorial arguments proposed here. The unique modification is
that here we will sometimes 
consider reductions from some $X \subseteq \Bai$ to $\Bai$: given such an $X$,
a set of functions $\F$ with domain $X$, and sets $A, B \subseteq \Bai$,  
we say that $A$
is \emph{$\F$-reducible} to $B$ (in symbols $A \leq_\F B$) just in case
there is some $f \in \F$ such that \emph{for every $x \in X$}
\[ x \in A\iff f(x)\in B,\]
that is $A \cap X = f^{-1}(B)$. 
Finally, given a set $\F$ of \emph{totally defined} functions, recall
that the \emph{Semi-Linear Ordering Principle} for $\F$  
(denoted by $\SLO^\F$) is 
the statement 
\[ \forall A,B \subseteq \Bai ({A \leq_\F B} \vee {\neg B \leq_\F
  A}),\]
where $\neg B$ denotes  $\Bai \setminus B$.

\section{Reduction games}\label{sectionreductiongames}

As recalled in the introduction, the first examples of reduction games
are  $G_\L$, $G_\W$, $G_\bt$  and $G_{\sf{E}}$ (defined in  
\cite[pp.\ 72 and 64]{wadgethesis}, \cite[p.\ 86]{vanwesepthesis}, and
\cite[p.\ 69]{duparc}, respectively), which represent,
respectively, the classes of functions $\L$, $\W$, $\sf{D}_2$ and
$\B_1$ (see e.g.\ \cite[Theorem B8]{wadgethesis} and \cite[Theorems 3.1, 4.1,  and
5.1]{semmesmultitapegames}). Here we just give a brief and informal
description of the rules of these games.
In $G_\L$, both $\pI$ and  
$\pII$ have to play a natural number at each of their turns. 
The game $G_\W$ is a variation of $G_\L$ in which $\pII$ has the
further option 
of ``passing'' (i.e.\ skipping her move at some turn), but with the condition
that at the end 
of the run she has played infinitely many natural numbers, i.e.\  she
has enumerated a real.
The backtrack game $G_\bt$ is a further variation of $G_\W$ in
which $\pII$ can still pass (with the same condition above), but even
\emph{backtrack},
i.e.\ she can delete all her previous moves at once and start to play
natural numbers (or pass) anew,  
with the restriction that in each run she can use this option only
finitely many times (this guarantees  
that at the end of the run she has indeed played some real). Finally,
the eraser game $G_{\sf{E}}$  can be described in the following way:
$\pI$ must play a natural number on each of his turn, while
$\pII$ can either play a natural number or erase the last natural number
which appears on her board, but to guarantee that at the end of each
run $\pII$ has indeed played a real,
 we require  
that for each $x \in X$ and each $n \in \omega$ there must be some $m$
such that for every $k \geq m$ we have $\leng(t_k) \geq n$, where $t_k$
is the sequence of natural numbers that $\pII$ has played (after
possible erasings) when $\pI$ has enumerated $\restr{x}{k}$. In other
words, $\pII$ has to enumerate a real $y \in \Bai$ and she has the
option of changing the $n$-th digit of $y$ at any time, but for each
$n$ she can take this option only finitely many times. 

Many other games (both old and new) can be obtained by modifying one of these
games, for example: 

\begin{itemize}

\item a simple variation of $G_\W$
 leads to the game $G_{k\text{-}\Lip}$, 
in which $\pII$ can still pass but \emph{at most} $k$ times in a run (this is also
equivalent to requiring that $\pII$ pass for the first $k$ turns and then
 plays the rest of the game without passing any more\label{kLip}): as
 we will see in Proposition
\ref{theorLipstrategies}, the strategies
 for $\pII$ in $G^X_{k\text{-}\Lip}$ 
induce exactly the functions in $\Lip(2^k)$; 

\item a variation of $G_\bt$ leads instead to the game
$G_{\Lip\text{-}\bt}$ in which $\pII$ can still backtrack finitely
many times but is no more allowed to pass: legal strategies for
$\pII$ in $G_{\Lip\text{-}\bt}$ induce exactly the functions in
$\sf{D}^\Lip_2$ (which is a proper
subset of $\sf{D}_2$ --- see \cite[pp.\ 45-46]{mottorosborelamenability});

\item another variation of $G_\W$ leads to the multitape game
  $G_\sf{M}$ defined 
  by Semmes in \cite[p.\ 202]{semmesmultitapegames}, which  can be described as
  follows: $\pI$ has to play his numbers on a single rows, while $\pII$ has
  to play her numbers on a table containing $\omega$-many rows. At each 
  turn, $\pI$ plays 
  a natural number on his unique row, while $\pII$ has first to choose one
  of her rows and then either pass or play a new natural number on it, with the
  condition that in each run of the game she has to 
  play an infinite amount of  numbers on exactly one of her rows: by
  Andretta-Semmes' \cite[Theorem 6.1]{semmesmultitapegames}, legal
  strategies for $\pII$ in this game induce the functions in 
  $\tilde{\sf{D}}_2$;

\item the game $G_{\sf{E}}$ can be ``iterated'' by using $n$ eraser
  operators ranked with a priority in order to obtain
games $G_{\B_n}$ whose legal strategies for $\pII$ induce exactly  
the Baire class $n$ functions.
\end{itemize}

Various other reduction games, like the
multitape eraser game $G_{\sf{ME}}$ defined in
\cite[Section 7]{semmesmultitapegames} which
represents $\sf{D}^{\B_1}_2$, can be obtained in a similar way.
However, each of 
these games seems 
to be strictly related to the particular presentation of the game
itself, and therefore it seems difficult to guess which should be the
definition of a game representing a more
complex class of functions, like e.g.\ $\tilde{\sf{D}}^\W_{\omega^{\omega+3}}$,
$\sf{D}^{\B_{28}}_{\omega^2+9}$, or 
$\F_{\bSigma^1_7}$. In order to have a uniform
approach to the problem of representing classes of functions by means of games,
it is useful to first abstractly define the notion of \emph{reduction
  game} (see Subsection \ref{sectionredg}), to isolate some basic examples of such games (like
the already defined $G_\W$, or games representing class of the form
$\F_\bGamma$ --- see Subsection \ref{sectionexamples} and Section
\ref{sectionGamma}), and then find some operations 
which correspond to the analogous topological operations used in the
definition of the new classes of functions, like  the operation of taking
pointwise limits, or of giving piecewise definitions on a definable
partition of the space (this is 
done in Sections \ref{sectionplayable} and \ref{sectionsmallness}).

\subsection{Reduction games and playable sets of functions}
\label{sectionredg}

 A reduction game is a tuple $G_* = (X,M_*,R_*,\iota_*)$ (where $*$ is
some symbol which identify the game) such that $X \subseteq \Bai$,
$M_*$ is a countable set disjoint from $\omega$ (called \emph{set of
  moves}), $R_* \subseteq \Bai \times {}^\omega (\omega \cup M_*)$ (called
\emph{set of rules}) and $\iota_*$ is a function from $R_*$ into
$\Bai$ (called \emph{interpretation function}). 

The rules of the game are as follows: $\pI$ plays elements of
$\omega$, while $\pII$ plays elements of $\omega \cup M_*$, so that
after $\omega$-many turns $\pI$ will have produced a real $x \in \Bai$
while $\pII$ will have produced $y \in {}^\omega(\omega \cup M_*)$
($y$ is called \emph{complete play} of $\pII$). Then we give the
following two conditions:

\begin{enumerate}[(1)]
 \item if $x \notin X$ then $\pI$ loses;
\item if $x \in X$ but   $(x,y) \notin R_*$ then $\pII$ loses.
\end{enumerate}

If $\sigma$ is a strategy for $\pI$ and $y$ is the complete play of
$\pII$ in a run of $G_*$, then $\sigma * y$ denotes the real
enumerated by $\pI$ while following $\sigma$ against $y$. 
Similarly, if $\tau$ is a strategy for $\pII$
and $x \in \Bai$, 
we denote by $x * \tau$ the complete play produced by $\pII$ in the run of 
$G_*$ in which $\pI$ 
enumerates $x$ and $\pII$ plays according to $\tau$
($\sigma * t$ and  $s * \tau$ are defined in a similar way for every
$t \in {}^{< \omega}(\omega \cup M_*)$ and $\emptyset \neq s 
\in \seqo$). A strategy $\sigma$ for $\pI$ is said to be \emph{legal}
if $\sigma * y \in X$ for every $y \in {}^\omega(\omega \cup M_*)$,
while a strategy $\tau$ for $\pII$ is said to be \emph{legal} if $(x,x *
\tau) \in R_*$ for every $x \in X$. The collection of legal strategies
for $\pII$ in $G_*$ will be denoted by $\sf{LS}_*$. 

Given $A,B \subseteq \Bai$, $G_*(A,B)$ is defined by the following
winning condition: If neither (1) or (2) have occurred, then $\pII$
wins if and only if $x \in A \iff \iota_*(x,y) \in B$ (in this case
$A,B$ are called \emph{payoff sets} of $G_*(A,B)$ and $\iota_*(x,y)$ is
called \emph{play} or \emph{output real} of $\pII$). 
A strategy (for either $\pI$ or $\pII$) is said to be
\emph{winning} in the game $G_*(A,B)$ if it is legal and always
guarantees the victory of the 
corresponding player, whatever his or her opponent plays. 

Notice that every $\tau \in \sf{LS}_*$ canonically induces the unique 
function
\[f_\tau \colon  X \to \Bai \colon  x \mapsto \iota_*(x,x*\tau),\]
(in this case we will say that $\tau$ \emph{represents} $f$), whereas
for some function $f \colon X \to \Bai$ there can be  
distinct $\tau,\tau' 
\in \sf{LS}_*$  such that $f = f_\tau =
f_{\tau'}$. We will put $\F_* = \{ f \colon X \to \Bai \mid f
= f_\tau \text{ for some }\tau \in \sf{LS}_*\}$. With this notation,
$\pII$ has a 
\emph{winning} strategy in the game $G_*(A,B)$ if and only if $A$
is \emph{$\F_*$-reducible} to $B$: this
is why the games described above are called \emph{reduction}
games.

Usually, the topological definition of a certain class of functions
is virtually independent from the particular domain of the functions
under consideration (apart from its topology, of course), meaning that
the definition of such class uses $X$ just as a sort of
parameter. For example, 
a function is said to be continuous if the preimage of an open set is
still open, and this definition does not involve any other information
on the domain of the function except for its topology. Therefore, to
have a decent notion of \emph{representation} of a (topologically
defined) class  of functions $\F$ (with domain arbitrary subsets of
$\Bai$) by means of a certain \emph{set} of
reduction games  $\mathcal{G}_* = \{ G_* = (X, M_*, R_*, \iota_*) \mid
X \subseteq \Bai\}$, it seems natural to stipulate by convention that
all the games in $\mathcal{G}_*$ share the same set of moves, set of
rules and interpretation function, or at least\footnote{We will take
  this second option just in Section \ref{sectionGamma}, when
  considering games related to boldface pointclasses $\bGamma$
  larger than the collection of Borel sets.} that all sets of rules
 (and consequently all the interpretation
functions) of the games in $\mathcal{G}_*$ are defined by a single
formula which uses 
$X$ as a parameter. Such classes of games are said to be
\emph{parametrized}. With a little abuse of  
notation, if $\mathcal{G}_*$ is a parametrized class of games we will
denote by 
$\F_*$ again the collection of all functions induced by the legal strategies
\emph{in one of the games of $\mathcal{G}_*$}.

\begin{defin}
  Let $\F$ be any set of functions from subsets of $\Bai$ into
  $\Bai$. We say that $\F$ is \emph{playable} if   there is 
  a parametrized class $\mathcal{G}_*$ of reduction games such  that $\F =
  \F_*$. In this case, we also say that the class $\mathcal{G}_*$
  \emph{represents} $\F$. 
\end{defin}

Finally, for $G_* = (\Bai,M_*, R_*, \iota_*)$ we will denote by
$\AD(G_*)$ (or simply $\AD^*$) the 
statement: ``for all $A,B \subseteq \Bai$, the game $G_*(A,B)$ is
determined (i.e.\ either $\pI$ or $\pII$ has a winning
strategy)''. $\AD^*$ is obviously a consequence of $\AD$ (to see this
it is enough to ``code'' in the natural way the reduction game $G_*$
into a classical game on $\omega$). Moreover, if $R_*$ and $\iota_*$
are not too complicated, the above implications have also ``local
versions''. For example, if $R_*$ is a Borel subset of
$\Bai \times {}^\omega(\omega \cup M_*)$ (endowed with the product 
topology) and $\iota_*$ is Borel, then Borel determinacy is sufficient
to have that for every \emph{Borel} $A,B \subseteq \Bai$ the game
$G_*(A,B)$ is determined (and, more generally, local versions of $\AD$
imply local versions of $\AD^*$). As we will see, if $\F$ is a
playable subset of the Borel functions it is in practice always the
case that the parametrized class $\mathcal{G}_*$ of reduction games 
which represents $\F$ has a Borel set of rules and a Borel interpretation
function (defined independently from $X$).

\subsection{Some examples} \label{sectionexamples}

Now we want to give some examples on how to formalize the games
presented at the beginning of this section into reduction games. As the
names suggest, the set $M_*$ will be used to code the alternative
moves  (like ``pass'', ``backtrack'', ``erase'', and so on) of $\pII$,
$R_*$ will be used to code the rules of the game (that is the rules
that $\pII$ must respect in order to have a chance of victory), and
$\iota_*$ will be used to recover from the (play of $\pI$ and the)
complete play of $\pII$ the 
real that must be used in checking the winning condition. We will present just
three cases, namely  continuous functions, Baire class
$1$ functions, and Lipschitz functions with constant $2^k$: however, we
will prove that the corresponding reduction game really represents the
desired class of functions just for the last case, as the other two
well-known proofs can be obtained using classical arguments (see e.g.\
\cite[Theorems 3.1 and 5.1]{semmesmultitapegames}). 

\begin{example}\label{exW}
Define the reduction game $G_\W = (X, M_\W, R_\W, \iota_\W)$ by:

- $M_\W  = \{ \p \}$ (the symbol $\p$ will be interpreted as ``pass'');

- $R_\W  = \{  (x,y) \in \Bai \times {}^\omega (\omega \cup \{ \p \})
\mid \forall n 
\exists m \geq n (y(m) \neq \p \})$; 

- $\iota_\W \colon R_\W \to \Bai \colon (x,y) \mapsto \langle y(n) \mid
y(n) \neq \p , n \in \omega \rangle$. 
\end{example}

\begin{proposition}
 A function $f \colon X \to \Bai$ is continuous if and only if $f =
 f_\tau$ for some $\tau \in \sf{LS}_\W$. 
\end{proposition}

\begin{example}\label{exE}
Define the reduction game $G_\sf{E} = (X, M_\sf{E}, R_\sf{E},
\iota_\sf{E})$ by: 

- $M_\sf{E}  = \{ \sf{E} \}$ (the symbol $\sf{E}$ will be interpreted
as ``erase'', and will correspond to the backspace key of a usual
computer keyboard); 

- for $s \in {}^{< \omega}(\omega \cup \{ \sf{E} \})$, inductively
define $\hat{\iota}(s) \in {}^{< \omega} \omega$ by letting
$\hat{\iota}(\emptyset) = \emptyset$, $\hat{\iota}(s^\smallfrown n) =
\hat{\iota}(s)^{\smallfrown}n$ (for $n \in \omega$), and
$\hat{\iota}(s^\smallfrown \sf{E}) = \hat{\iota}(s) \restriction
(\leng(\hat{\iota}(s))-1)$;  

- $R_\sf{E}  = \{ (x,y) \in \Bai \times {}^\omega (\omega \cup \{ \sf{E}
\}) \mid \forall 
n \exists m \forall k \geq m (\leng(\hat{\iota}(y \restriction k))
\geq n \})$; 

- $\iota_\W \colon R_\W \to \Bai \colon (x,y) \mapsto \bigcup_n
\hat{\iota}(y \restriction k_n)$, where $k_n$ is the smallest $m$ such that
$\forall k \geq m (\leng(\hat{\iota}(y \restriction k))
\geq n$. 
\end{example}

\begin{proposition}
 A function $f \colon X \to \Bai$ is of Baire class $1$ if and only if
 $f = f_\tau$ for some $\tau \in \sf{LS}_\sf{E}$.
\end{proposition}

\begin{example}
Define the reduction game $G_{k\text{-}\Lip} = (X, M_{k\text{-}\Lip},
R_{k\text{-}\Lip}, \iota_{k\text{-}\Lip})$ by\footnote{Notice that
  here we are using the (equivalent) definition of $G_{k\text{-}\Lip}$
  as the game in which $\pII$ pass exactly for the first $k$ turns ---
  see page \pageref{kLip}.}:

- $M_{k\text{-}\Lip} = \{ \p \}$ (the symbol $\p$ will be interpreted
as ``pass'' again); 

- $R_{k\text{-}\Lip} = \{ (x,y) \in {}^\omega (\omega \cup \{ \p \}) \mid
\forall n ({y(n) = \p} \iff {n < k})  \})$; 

- $\iota_{k\text{-}\Lip} \colon R_{k\text{-}\Lip} \to \Bai \colon (x,y)
\mapsto \seq{y(n+k)}{n \in \omega}$.
\end{example}

\begin{proposition}\label{theorLipstrategies}
For every $k \in \omega$,
a function $f \colon X \to \Bai$ is in $ \Lip(2^k)$ if and only if
there is some $\tau \in \sf{LS}_{k\text{-}\Lip}$
such that $f = f_\tau$.
\end{proposition}

\begin{proof}
Let $\tau \in \sf{LS}_{k\text{-}\Lip}$. Then
$\seq{(s*\tau)(n+k)}{n+k 
  < \leng(s)}$ has length  $\max \{ 0,
\leng(s)-k\}$, which
implies that $d(f_\tau(x),f_\tau(x')) \leq 2^k d(x,x')$ for every $x,x' \in X$.
For the other direction, let $f \colon X \to \Bai$ be in
$\Lip(2^k)$. Given $s \in  {}^{n+k} \omega$, let $t_s \in
{}^n \omega$ be the unique sequence such that $\bN_s \cap X \subseteq
f^{-1}(\bN_{t_s})$ (such a sequence must exists since $f \in
\Lip(2^k)$), and define the strategy $\tau$ for $\pII$ in
$G_{k\text{-}\Lip}$ by  $\tau(s) = \p$ if $\leng(s) < k$ and 
$\tau(s)  = t_s(\leng(t_s) - 1)$ otherwise. 
It is clear that $\tau$ is legal and such that $f = f_\tau$.
\end{proof}

\section{Constructing new reduction 
games} \label{sectionplayable}

Let us start with a technical definition.

\begin{defin}\label{defadequate}
 Let $G_* = (X,M_*,R_*,\iota_*)$ be a reduction game
 and $\F_*$ be the set of functions 
induced by the strategies in $\sf{LS}_*$. We say that $G_*$
(or $\F_*$) is \emph{$\p$-closed}   
if $\F_*$ coincides with the set of functions induced by the legal strategies
for $\pII$ in the new reduction 
game $G^\p_*$ defined by: 

- $M^\p_* = M_* \cup \{ \p \}$, where $\p$ is a new symbol not in
$\omega \cup M_*$; 

- $R^\p_* = \{ (x,y) \in \Bai \times {}^\omega (\omega \cup M^\p_*)
\mid {\forall n 
  \exists m \geq n (x(m) \neq \p)} \wedge  
{(x,\langle y(n) \mid y(n) \neq \p, n \in \omega \rangle) \in R_*} \}$; 

- $\iota^\p_* \colon R^\p_* \to \Bai \colon (x,y) \mapsto
\iota_*(x,\langle y(n) 
\mid y(n) \neq \p, n \in \omega \rangle)$. 

A set of functions $\F$ is \emph{adequate} if it contains the identity
function and there is a parametrized class $\mathcal{G}_*$ of $\p$-closed 
reduction games which represents $\F$. 
\end{defin}

Roughly speaking, the condition for a set $\F_*$ of being $\p$-closed
is the natural counterpart in terms of strategies of the property of
being closed under right-composition with continuous functions from
$X$ into itself, while a set of functions is adequate if it is not too
small.  
The previous definition could seem a little bit obscure, but it covers
almost all the important cases and is designed in such a way that the
arguments presented in the next subsections can be carried out in a very
general way\footnote{Notice that, as shown in Section
  \ref{sectionsmallness}, adequateness is not the optimal condition
  for our purpose: nevertheless, it allows to give an easier
  presentation on the subsequent constructions avoiding some
  technical and notational complications, and thus seems to be a good
  compromise between generality of the arguments and clearness of
  exposition.}.   
For instance, it is obvious that $G_\W$ is $\p$-closed,  but let us
check as a nontrivial example that $G_{\sf{E}}$ (and hence $\B_1$) is 
$\p$-closed as well. Consider the game $G^\p_{\sf{E}}$: clearly
$\sf{LS}_{\sf{E}} \subseteq \sf{LS}^\p_{\sf{E}}$ 
and, conversely, every $\tau \in \sf{LS}^\p_{\sf{E}}$ can be
converted in a legal strategy for $\pII$ in $G_{\sf{E}}$ by
substituting every use of the symbol $\p$ with the pair of moves
``play $0$ and then play $\sf{E}$''. However, one has also to notice that
not all the sets of functions 
considered in this paper are $\p$-closed --- for a counterexample just
take $\Lip(2^k)$.

In the next subsections we will show how to construct games for the classes of
functions $\sf{D}^{\vec{\F}}_\xi$, $\tilde{\sf{D}}^{\vec{\F}}_\xi$ and
$\lim \vec{\F}$, provided that $\xi$ is some fixed countable nonzero ordinal
and $\vec{\F}$ is a countable sequence of not too small (i.e.\ adequate)
playable sets of functions (albeit for simplicity of presentation in
Subsections \ref{sectionpiecewise} and \ref{sectiontilde} we
will just deal  with the simpler case in which $\vec{\F}$ is
constantly equal to some fixed adequate $\F$: this covers all the most
important cases that one encounters in practice, and the
general case can 
easily be recovered from these particular examples). Taking $\F = \W$
in Theorems \ref{theorgames} and \ref{theorgames2} we get, in
particular, a generalization of the 
games $G_\bt$ and $G_\sf{M}$ for (simultaneously) \emph{all} higher
level (the existence of such games was still an open
problem). Moreover, since it is easy to see that the constructions
below always produce reduction games which are $\p$-closed, by iterating
those 
 constructions one can get a wide class of reduction games, namely
 games representing $\sf{D}^{\mathcal{B}_\mu}_\xi$ and 
 $\tilde{\sf{D}}^{\mathcal{B}_\mu}_\xi$ for every $\xi,\mu<\omega_1$.

\subsection{Games for $\sf{D}^\F_\xi$}\label{sectionpiecewise}

Fix any increasing sequence of ordinals $\langle \mu_n
\mid  n \in \omega \rangle$ cofinal in $\xi$ and, for each $n \in
\omega$, a set $P_n$ which is $\bPi^0_{\mu_n}$-complete (we will see
in Claim \ref{claimchangecontrolsets}
that the choice of the
$\mu_n$'s and of the $P_n$'s is not essential).  Let $\F$ be an  adequate set
of functions,  and let  $G_* = (X , M_*, R_*, \iota_*)$ be a
$\p$-closed reduction game representing the subset of $\F$ consisting
of those function which have domain $X$. Let
$G_\W  = (\Bai, M_\W, R_\W, \iota_\W)$ be the game defined in
Example \ref{exW} representing
the set of (totally defined) continuous functions $\W$.   Now define
$G^\F_\xi = (X, M^\F_\xi, R^\F_\xi,
\iota^\F_\xi)$ as follows: 

- $M^\F_\xi = M_* \cup M_\W$;

- $R^\F_\xi = \{ (x,y) \in \Bai \times {}^\omega (\omega \cup
M^\F_\xi) \mid {\forall i ((x,{\pi_{2i}(y)) \in R_\W} \wedge
  {(x,\pi_{2i+1}(y)) \in R_*})} \wedge {\exists i (\iota_\W(x,\pi_{2i}(y))
  \in P_i)} \}$; 

- $\iota^\F_\xi \colon R^\F_\xi \to \Bai \colon (x,y)
\mapsto \iota_*(x,\pi_{2i+1}(y))$, where $i$ is smallest such that
$\iota_W(x,\pi_{2i}(y)) \in P_i$.

The
game $G^\F_\xi$ can be visualized as a two-player game in which
$\pI$ has to fill in a table with a single row, while $\pII$ has to
fill in a \emph{table with $\omega$-many rows}.

\begin{small}
\begin{center}
\begin{tabular}{@{}cc|cccccccccccc@{}}
&&&&&&&&&&&&&\\
\pI& $A$& $x_0$
&$x_1$&$x_2$&$\dotsc$&$\dotsc$&$x_k$&$x_{k+1}$
&$\dotsc$&&&&$\rightarrow x$\\
&&&&&&&&&&&&&\\
\hline
&&&&&&&&&&&&&\\
&$P_0$ &$c^0_0$&$c^0_1$ &$c^0_2$&
$\dotsc$&$\dotsc$&$c^0_k$&$\dotsc$&&&&&$\rightarrow c^0$\\
&&&&&&&&&&&&&\\
&
B&$y^0_0$&$y^0_1$&$y^0_2$&$\dotsc$&$\dotsc$&$y^0_k$&$\dotsc$&&&
&&$\rightarrow y^0$\\
&&&&&&&&&&&&&\\
&$P_1$&$c^1_0$&$c^1_1$&$c^1_2$
&$\dotsc$&$\dotsc$&$c^1_k$&$\dotsc$&&&&&$\rightarrow c^1$\\
&&&&&&&&&&&&&\\
\pII&$B$
&$y^1_0$&$y^1_1$&$y^1_2$&$\dotsc$&$\dotsc$&$y^1_k$&$\dotsc$&&&
&&$\rightarrow y^1$\\
&&&&&&&&&&&&&\\
& &$\vdots$&&&&&&&&&&&\\
&&&&&&&&&&&&&\\
 &$P_n$
 &$c^n_0$&$c^n_1$&$c^n_2$&$\dotsc$&$\dotsc$ &$c^n_k$&
 $\dotsc$&& &&& $\rightarrow c^n$\\
&&&&&&&&&&&&&\\
& $B$
&$y^n_0$&$y^n_1$&$y^n_2$&$\dotsc$&$\dotsc$&$y^n_k$&$\dotsc$&&&&&
$\rightarrow y^n$\\
&&&&&&&&&&&&&\\
&&$\vdots$&&&&&&&&&&&
\end{tabular}
\end{center}
\end{small}

At the $k$-th turn, $\pI$ plays a natural number on his (unique) row,
while player $\pII$ has two options (in what follows
 $n$ is the unique natural number such that $\langle n,m \rangle = k$
 for some/any $m \in \omega$): pass or play a natural number on
the $n$-th row if $n=2i$ is even (but at the
end of the run she must have played infinitely many natural numbers on
such row,
i.e.\ she must have produced a real $c^i$ on it),  or else play
an element of $\omega \cup M_*$  on 
 her $n$-th rows if $n$ is odd.  
The even rows are \emph{control} rows which can \emph{activate} the rows
immediately below them (the odd ones), and this happens exactly when
the real $c^i$
played on the $2i$-th row  belongs to the \emph{control set} $P_i$. In
every run of the game, $\pII$ has to activate \emph{at least} one of
the odd rows, and she has to make sure that the sequence $y^i$ she has
played on the $2i+1$-st row belongs to the set of rules
$R_*$ for every $i \in \omega$ (i.e.\ she must ``respect the rules'' of 
$G_*$ on the odd rows). Finally, the output 
real played by $\pII$ is exactly $\iota_*(x,y^i)$,
where $i$ is least such that the $2i+1$-st row is
activated.\\

Every strategy $\tau$ for $\pII$ in $G^\F_\xi$ can be seen as a
sequence $\langle \tau_n\mid n \in \omega \rangle$ of legal strategies for
$\pII$ in the game $G_\W= (\Bai,M_\W,R_\W,\iota_\W)$ or in $G_*=
(X,M_*,R_*,\iota_*)$ (depending on whether $n$ 
is  even or odd),
each of which is used on the  
corresponding row. In fact, if $\langle \tau_n\mid n \in\omega
\rangle$ is such a  sequence we can define
the strategy $\tau = \bigotimes_n \tau_n$ for $\pII$ in $G^\F_\xi$
as follows: 
for each $\emptyset \neq s \in \seqo$, let $n,m \in \omega$ be such
that $\leng(s) = \langle n,m \rangle +1$, and define $\tau(s) =
\tau_n(s \restriction (m+1))$. 
 It is not hard to check that $\tau = \bigotimes_n \tau_n$ is  a (non
 necessarily legal) 
strategy for $\pII$ in $G^\F_\xi$ such that $\pi_n(x*\tau) = x *
\tau_n$ for every $x \in X$. Thus to define a legal strategy for
$\pII$ in $G^\F_\xi$ it is enough to give 
a sequence $\langle \tau_n \mid n \in \omega \rangle$ of
legal strategies for
$\pII$ in  $G_\W$ (resp.\ $G_*$) if $n$ is even (resp.\
odd),
and check that for every $x \in X$ there is 
some $i \in \omega$ such that $\iota_\W(x,x * \tau_{2i}) \in P_i$.

Conversely, given a legal strategy $\tau$ for $\pII$ in
$G^\F_\xi$ and a 
natural number $n \in \omega$, we can define the legal strategy
$\pi_n(\tau)$ for $\pII$ in, respectively, $G^\p_\W$ if $n$ is even or in
$G^\p_*$ if $n$ is odd (where $G^\p_\W$ and $G^\p_*$ are defined as in
Definition \ref{defadequate}) as follows: for each $\emptyset \neq s
\in \seqo$, define $\pi_n(\tau) (s) = \p$ if there is no $m \in
\omega$ such that $\leng(s) = \langle n,m \rangle +1$, and $\pi_n(\tau)(s) =
\tau(s)$ otherwise. Since both  $G_\W$ and $G_*$
are $\p$-closed, with a little abuse of
notation we will confuse each $\pi_{2i}(\tau)$ (respectively,
$\pi_{2i+1}(\tau)$) with any legal strategy in $G_\W$ (resp.\ $G_*$)
which induces the same function $f_{\pi_{2i}(\tau)}$ (resp.\
$f_{\pi_{2i+1}(\tau)}$) on $X$. 
It is not hard to check that  the operations
$\pi_n$ on strategies ``commute'' with $\bigotimes_n$:  given a
strategy $\tau$ for 
$\pII$  in $G^\F_\xi$,
for every $x \in X$ and $i \in \omega$ we have that $\iota_\W(x,\pi_{2i}(x *
\tau)) = \iota_\W(x,\pi_{2i}(x * \bigotimes_n \pi_n(\tau)))$ and
$\iota_*(x,\pi_{2i+1}(x * 
\tau)) = \iota_*(x,\pi_{2i+1}(x * \bigotimes_n \pi_n(\tau)))$.
In particular,
 $\tau$ is a \emph{legal} strategy
for $\pII$ in $G^\F_\xi$ if and only if $\bigotimes_n \pi_n(\tau)$
is, and in  
the positive case the two strategies induce the same function on $X$ (i.e.\
$f_\tau = f_{\bigotimes_n \pi_n(\tau)}$). 
\\

We will now prove that, as already announced,  the choice of the
ordinals $\mu_n$'s and of the sets $P_n$'s is not essential. Let
$\langle \hat{\mu}_n\mid n \in \omega \rangle$ be a (non necessarily
increasing)
 sequence of
ordinals cofinal in  $\xi$, and for every $n \in \omega$ let
$\hat{P}_n$  be
$\bPi^0_{\hat{\mu}_n}$-complete. Let $\hat{G}^\F_\xi$  be the
game defined as 
at the beginning of this section but using the $\hat{P}_n$'s instead of the
$P_n$'s. 

\begin{claim}\label{claimchangecontrolsets}
For every legal strategy $\hat{\tau}$ for
$\pII$ in 
$\hat{G}^\F_\xi$ there is a legal strategy $\tau$ for $\pII$ in
$G^\F_\xi$ such  
that $f_{\hat{\tau}}= f_\tau$, and, conversely, for every legal
strategy $\rho$ for $\pII$ in $G^\F_\xi$ there is a legal strategy
$\hat{\rho}$ for $\pII$ in $\hat{G}^\F_\xi$ such that $f_\rho =
f_{\hat{\rho}}$. Therefore, $G^\F_\xi$ and $\hat{G}^\F_\xi$ represent
the same set of functions. 
\end{claim}

\begin{proof}
Since the $\mu_n$'s are cofinal in
$\xi$ and the $P_n$'s 
are $\bPi^0_{\mu_n}$-complete, for every $k \in \omega$ there is some
$n_k$ such that $\hat{P}_k \leq_\W P_{n_k}$, thus
there is a winning strategy $\sigma_k$ for $\pII$ in
$G_\W(\hat{P}_k,P_{n_k})$.  Moreover, we can fix some $y_n \notin P_n$ for
every $n$, and for every $y \in \Bai$ let $\rho_y \in \sf{LS}_\W$
 and $\rho_\id \in \sf{LS}_*$ be 
 such that $f_{\rho_y}$ is constantly equal to $y$ and $f_{\rho_\id} =
 \id$
 is the identity function. Finally, given $\tau_0 ,
 \tau_1 \in \sf{LS}_\W$  let $\tau_1 \star \tau_0 \in \sf{LS}_\W$ be
 any strategy such that $f_{\tau_1 \star \tau_0} = f_{\tau_1} \circ
 f_{\tau_0}$. 
 Now define
 $\tau_{2n} =  \sigma_k \star \pi_{2k}(\hat{\tau})$ and $\tau_{2n+1} =
 \pi_{2k+1}(\hat{\tau})$ 
 if $n = n_k$ for some $k \in \omega$, and
 $\tau_{2n} = \rho_{y_n}$ and $\tau_{2n+1} = \rho_\id$
 otherwise. Finally put $\tau= \bigotimes_n \tau_n$. 

Given  $x \in \Bai$, it is not hard to
check that $ \iota_\W(x,\pi_{2k}(x * \hat{\tau})) \in \hat{P}_k \iff
\iota_\W(x,\pi_{2n_k}(x 
* \tau)) \in P_{n_k}$
and $\iota_*(x,\pi_{2k+1}(x * \hat{\tau})) =
\iota_*(x,\pi_{2n_k+1}(x*\tau))$,  while $\iota_\W(x,\pi_{2n}(x
* \tau))= y_n \notin P_n$ for every $n$ which is not of the form $n_k$
for some $k \in \omega$. Hence $\tau$ is a legal strategy for $\pII$ in
$G^\F_\xi$ if and only if $\hat{\tau}$ were a legal strategy
for $\pII$ in 
$\hat{G}^\F_\xi$, and moreover for every $x \in \Bai$ we have
that $\iota_*(x,\pi_{2k+1}(x * 
\hat{\tau})) = \iota_*(x,\pi_{2n_k+1}(x * \tau))$, where $k$ is least such that
$\iota_\W(x,\pi_{2k}(x*\hat{\tau})) \in \hat{P}_k$ (which implies that
$n_k$ is the least $m 
 \in \omega$ such
that $\iota_\W(x,\pi_{2m}(x*\tau)) \in P_m$).

Using the same argument, one can convert any $\rho \in
\sf{LS}^\F_\xi$ into a $\hat{\rho} \in \hat{\sf{LS}}^\F_\xi$ such that
$f_{\hat{\rho}} = f_\rho$, hence we are done. 
\end{proof}

Thus from this point onward we can take the option of changing the
sets $P_n$'s in the definition of $G^\F_\xi$ at our pleasure, provided
that
the new ones are
$\bPi^0_{\hat{\mu}_n}$-complete for some (non necessarily increasing)
sequence of ordinals $\seq{\hat{\mu}_n}{n \in \omega}$ cofinal in $\xi$.

\begin{theorem}\label{theorgames}
For every $X,A,B \subseteq \Bai$ and every $f\colon X \to \Bai$ we
have that:
\begin{enumerate}[i)]
\item $f \in \sf{D}^\F_\xi$ if and only if there is some $\tau \in
  \sf{LS}^\F_\xi$ such that   $f=f_\tau$;
\item if $\pI$ has a winning strategy in $G^\F_\xi(A,B)$, then $\pI$
  has also a winning strategy in $G_\L(A,B)$. In particular, in this
  case, $B \leq_\sf{c} \neg A$, i.e.\ there is 
  a \emph{contraction} (that is a Lipschitz function with constant $< 1$) $g$ 
  such that $x \in B \iff g(x) \notin A$ for every $x \in \Bai$, and
  ${\rm range}(g) \subseteq X$. 
\end{enumerate}
\end{theorem}

\begin{proof}
  First suppose that $f\colon X \to \Bai$ is in $\sf{D}^\F_\xi$: let
  $\langle D_k\mid k \in \omega \rangle$ be a sequence of
  $\bPi^0_{<\xi}$-sets such that $\seq{D_k \cap X}{k \in \omega}$ is a
  partition  of
  $X$, and $\seq{f_k}{k \in \omega}$ be a sequence of continuous
  functions from $X$ into $\Bai$ such that
  $\restr{f}{D_k}=\restr{f_k}{D_k}$ for every $k 
  \in \omega$. 
  By definition of the sets $P_n$'s, 
  we can find an
  increasing sequence of natural numbers $n_k$ such that  $\pII$ has a
  winning strategy  
$\sigma_k$ in
  $G_\W(D_k,P_{n_k})$ for each
  $k \in \omega$. Moreover, we can choose the reals $y_n \notin P_n$
  and define the 
  strategies $\rho_y$ and $\rho_\id$ as in the proof of Claim
  \ref{claimchangecontrolsets}. Finally, since $f_k \in \F$  for
  each $k \in \omega$, we can find
 strategies $\hat{\tau}_k \in \sf{LS}_*$ 
such that $f_{\hat{\tau}_k} = f_k$, i.e.\ such that
  $f_k(x) = \iota_*(x,x * \hat{\tau}_k)$ for every $x \in X$. Now put
$\tau_{2n} = \sigma_k$ and 
$\tau_{2n+1} =  \hat{\tau}_k$
if $n = n_k$ for some $k \in \omega$, and
$\tau_{2n} =  \rho_{y_n}$ and $\tau_{2n+1} = \rho_\id$
otherwise. Clearly $\tau = \bigotimes_n \tau_n$ is a legal
strategy for $\pII$ in $G^\F_\xi$. Moreover, for every $x \in X$
there is a unique $k$ such that $x \in D_k$, so
that $\iota_\W(x,\pi_{2n}(x*\tau)) \in P_n$  just for $n=n_k$: thus
for every $k \in \omega$ and $x \in D_k$
\[f_\tau(x)=\iota_*(x,\pi_{2n_k+1}(x* \tau)) =
\iota_*(x,x * \hat{\tau}_k) = f_k(x) = f(x).\]

Conversely, given a legal strategy for $\pII$ in $G^\F_\xi$, define
\begin{align*}
F_0  = &\{x \in X \mid \iota_\W(x,\pi_0(x * \tau)) \in P_0\}\\ 
F_{n+1}= &
\{ x \in X\mid  {\iota_\W(x,\pi_{2(n+1)}(x*\tau)) \in P_{n+1}} \wedge
{\forall i \leq n 
(\iota_\W(x,\pi_{2i}(x*\tau)) \notin P_i})\}.
\end{align*}
 Clearly the $F_n$'s form
a $\bDelta^0_\xi$-partition of $X$ and $\pi_{2n+1}(\tau) \in \sf{LS}_*$
for every $n$. Thus each $\pi_{2n+1}(\tau)$ induces a function
$f_n = f_{\pi_{2n+1}(\tau)}\colon  X \to \Bai$ in $\F$, and it is easy
to check that 
$\restr{f_\tau}{F_n} = \restr{f_n}{F_n}$ for every $n \in \omega$,
that is $f_\tau \in \sf{D}^\F_\xi$.

Finally, let $\rho$ be a winning strategy for $\pI$ in
$G^\F_\xi(A,B)$. We  define a strategy
$\sigma$ for $\pI$ in $G_\L(A,B)$ in the following way:
\begin{quote}
  Let $y \in \Bai$ be the real enumerated by $\pII$ in
  $G_\L(A,B)$. Consider the run  of the auxiliary
  game $G_0 = G_*(\Bai, \Bai)$ in which $\pI$ enumerates $y$  and
  $\pII$ follows $\rho_\id$. Now fix $z \in P_0$, and consider the run
  of a second 
  auxiliary game $G_1 = G^\F_\xi(A,B)$ in which  
 $\pI$ follows $\rho$  and  $\pII$  uses $\rho_z$ on the
  even rows, and ``copy'' the moves of $\pII$ in the previously described
  run of $G_0$ 
  on the odd ones 
  (the strategy for $\pII$ 
  defined in this way 
  is clearly legal since $z \in P_0$). Then at each turn ``copy''  the
corresponding  move made by $\pI$ in the run of
 $G_1$ described above.
\end{quote}

It is not hard to check that since $\rho$ is winning then 
$\sigma * y \in X$
and $\sigma * y\in A \iff y \notin
B$:
thus $\sigma$ is a winning strategy for $\pI$ in $G_\L(A,B)$. The rest
of part \textit{ii)} follows by standard arguments.
\end{proof}

\begin{remark}\label{remsemmes}
Although the game $G^\F_\xi$ and the multitape game $G_\sf{M}$ defined in
\cite{semmesmultitapegames} were developed independently, one should
notice that the easiest
direction of the proof of Theorem 
\ref{theorgames} (the one which goes from strategies to functions)
presents some affinity (at least in spirit) with the corresponding direction of the proof of
\cite[Theorem 6.1]{semmesmultitapegames}. However, the definition of
$G_\sf{M}$ (and, consequently, the whole Theorem 6.1 of \cite{semmesmultitapegames})
does not seem to admit a simple and straightforward generalization for
higher levels: this 
should be contrasted with the definition of the games $G^\W_\xi$ (and
of the games $\tilde{G}^\W_\xi$ defined in the next subsection), which
simultaneously gives a sort of ``uniform'' definition for all levels
$\xi$ of games 
representing $\sf{D}^\W_\xi$ and $\tilde{\sf{D}}^\W_\xi$  
--- in fact the games $G^\W_\xi$  and $\tilde{G}^\W_\xi$ can be
seen as a direct ``translation'' of the definitions 
of $\sf{D}^\W_\xi$ and  $\tilde{\sf{D}}^\W_\xi$ into the
game-theoretic formalism. 
\end{remark}

\subsection{Games for $\tilde{\sf{D}}^\F_\xi$}\label{sectiontilde}

We now want to prove that also the collection $\tilde{\sf{D}}^\F_\xi$
is playable by showing how to modify the game $G^\F_\xi$ to
obtain
the game $\tilde{G}^\F_\xi$, which will represent this new set 
of functions. The idea is to allow $\pII$ to not follow the rules on
some of her 
rows. Here is the formal definition of $\tilde{G}^\F_\xi = (X,
\tilde{M}^\F_\xi, \tilde{R}^\F_\xi, 
\tilde{\iota}^\F_\xi)$:

- $\tilde{M}^\F_\xi = M_* \cup M_\W$;

- $\tilde{R}^\F_\xi = \{ (x,y \in \Bai \times {}^\omega (\omega \cup
\tilde{M}^\F_\xi) \mid \forall i ((x,\pi_{2i}(y)) \in R_\W)  \wedge
\exists i (\iota_\W(x,\pi_{2i}(y) )
  \in P_i) \wedge
  \forall i ({\iota_\W(x,\pi_{2i}(y)) \in P_i} \imp (x,\pi_{2i+1}(y))
    \in R_*)  \}$;  

- $\tilde{\iota}^\F_\xi \colon \tilde{R}^\F_\xi \to \Bai \colon (x,y)
\mapsto \iota_*(x,\pi_{2i+1}(y))$, where $i$ is smallest such that
$\iota_W(x,\pi_{2i}(y)) \in P_i$.

Thus the game $\tilde{G}^\F_\xi$ can be visualized
as the variant of the game $G^\F_\xi$ in which $\pII$ must ``respect the
rules''  
\emph{just  on all the activated rows} (rather than on all her odd rows). 

Every strategy $\tau$ for $\pII$ in
$\tilde{G}^\F_\xi$
can again be seen as a product $\bigotimes_n \tau_n$ of strategies $\tau_n$
for $\pII$ in $G_\W= (\Bai,M_\W,R_\W,\iota_\W)$ and $G_*=(X_n, M_*, R_*,
\iota_*)$ (where now \emph{$X_n$ is a subset of $X$} 
depending on the index $n$). 
One can also define the
projections $\pi_n$ on  strategies as in the previous subsection, and
check that they ``commute'' with the operation $\bigotimes_n$. Note that
given a sequence of  strategies $\tau_n$ as above, $\bigotimes_n \tau_n$ is a
\emph{legal} strategy for $\pII$ in $\tilde{G}^\F_\xi$ if and only
if  $X_{2n+1} \supseteq f_{\tau_{2n}}^{-1}(P_n)$ 
and $\{f_{\tau_{2n}}^{-1}(P_n) \mid n \in
  \omega\}$ cover $X$. Using this fact 
one can prove the following theorem in a similar way to Theorem
\ref{theorgames}. 

\begin{theorem}\label{theorgames2}
For every $X,A,B \subseteq \Bai$ and every $f\colon X \to \Bai$ we
have that:
\begin{enumerate}[i)]
\item $f \in \tilde{\sf{D}}_\xi$ if and only if there is some $\tau \in
\tilde{\sf{LS}}^\F_\xi$ such that
  $f=f_\tau$;
\item if $\pI$ has a winning strategy in $\tilde{G}^\F_\xi(A,B)$,
  then $\pI$ 
  has also a winning strategy in $G_\L(A,B)$. 
\end{enumerate}
\end{theorem}

\subsection{Games for $\lim\vec{\F}$}\label{sectionlim}

The idea to require $\pII$ to fill a table with
$\omega$-many rows allows us also to define a (quite trivial) game for $\lim
\vec{\F}$ (hence, in particular, for all $\B_\xi$'s). Since in this case
considering an arbitrary sequence $\vec{\F} = \langle \F_n \mid n \in
\omega \rangle$ of adequate playable sets of functions does not
significatively increase the complexity of the 
presentation, we will not restrict ourselves to a constant
$\vec{\F}$. Suppose that the functions in $\F_n$ have all domain
$X$,
and let $G_n  = (X, M_n, R_n, \iota_n)$ be a sequence of $\p$-closed reduction
 games, each representing the corresponding $\F_n$.
The reduction game $G_{\lim \vec{\F}} = (X, M_{\lim \vec{\F}}, R_{\lim
  \vec{\F}}, 
 \iota_{\lim \vec{\F}})$ is defined as follows:

- $M_{\lim \vec{\F}} = \bigcup_n M_n$;

- $R_{\lim \vec{\F}} = \{ (x,y) \in \Bai \times {}^\omega (\omega \cup M_{\lim
  \vec{\F}}) \mid {\forall n ((x,\pi_n(y)) \in R_n)}
\wedge {\lim_n \iota_n(x,\pi_n(y)) \text{ exists}} \}$; 

- $\iota_{\lim \vec{\F}} \colon R_{\lim \vec{\F}} \to \Bai \colon (x,y)
\mapsto \lim_n \iota_n(x,\pi_n(y))$.

The game $G_{\lim \vec{\F}}$ can be visualized as the game in
  which at each turn $\pI$ must play a 
 natural number 
 on his (unique) row, while $\pII$ has to play either a natural number
 or a symbol from $M_n$ on the $n$-th row of her table with $\omega$-many rows,
with the condition that she must ``respect the rules'' of the
corresponding game on each of these
rows and that $\lim_n x_n$ must exists, where $x_n$ is the value of
$\iota_n$ on (i.e.\ the ``interpretation'' of) what $\pII$ has played
on the $n$-th row: in this case, the output real of $\pII$ is exactly
$\lim_n x_n$. 

As
for the games $G^\F_\xi$, it is easy to check that every strategy
$\tau$ for $\pII$ in $G_{\lim \vec{\F}}$ can be decomposed into $\omega$-many
strategies $\pi_n(\tau)$ for $\pII$ in $G_n$ (one for each row),
and conversely $\omega$-many strategies $\tau_n$ for $\pII$ in
$G_n$ can be coded up into a unique strategy $\bigotimes_n
\tau_n$ for $\pII$ in $G_{\lim \vec{\F}}$. Moreover it is easy to check that $f
\colon X \to \Bai$ is in $\lim \vec{\F}$ if and only if there is 
some $\tau \in \sf{LS}_{\lim \vec{\F}}$ such that $f = f_\tau$
(this is because the use of the table with $\omega$-many rows allows
to directly code the 
definition of ``being limit of a sequence of functions'' into a single
game).

\section{Games for $\bGamma$-measurable 
functions}\label{sectionGamma}

Let $\bGamma$ be any $\bSigma$-pointclass. The main goal of this
section is to construct games representing the collection $\F_\bGamma$ of
$\bGamma$-measurable functions $f \colon X \to \Bai$. When $\bGamma
= \bSigma^0_{\xi+1}$ (for some $\xi < \omega_1$) this just give an
alternative way of defining games for Baire class $\xi$ functions (see
Section \ref{sectionlim}), but note that
since e.g.\ $\bSigma^1_n$ is 
trivially a $\bSigma$-pointclass (for every $n \in \omega$), the main
result of this section gives also a new way (albeit less
informative than the construction 
given in \cite{semmesthesis})  of defining a
game for the class of all Borel functions (taking $\bGamma = \bSigma^1_1$), and
simultaneously 
solves the problem of 
finding games for the projective functions
posed by Semmes in his Ph.D.\ thesis \cite{semmesthesis}.


 Fix a universal set $S \subseteq
\Bai \times \Bai$ for
$\bGamma$ and let $G_\W = (\Bai,M_\W, R_\W, \iota_\W)$ be the Wadge game
representing (totally defined) continuous functions. Here is the
definition of  the game 
$G_\bGamma = (X, 
M_\bGamma, R_\bGamma, \iota_\bGamma)$:

- $M_{\bGamma} = M_\W = \{ \p \}$;

- $R_\bGamma = \{ (x,y) \in \Bai \times {}^\omega (\omega \cup
M_\bGamma) \mid {\forall 
n ((x,\pi_n(y)) \in R_\W)} \wedge 
{\forall n,m (D_{y,n,m} \cap X \in \bDelta_\bGamma(X))} \wedge
{\forall n \exists m (x 
\in D_{y,n,m})} \}$, where $D_{y,n,m} = \{ x \in \Bai \mid
(\iota_\W(x,\pi_{\langle n,m \rangle }(y)),x) \in S \} = \{ x \in \Bai \mid x \in S_{\iota_\W(x,\pi_{\langle n,m \rangle }(y))} \}$;

- $\iota_\bGamma \colon  R_\bGamma \to \Bai \colon (x,y) \mapsto
\iota_\bGamma(x,y)$, where $\iota_\bGamma(x,y)(n) = m \iff m$ is smallest
such that $x
\in D_{y,n,m}$.

The game $G_\bGamma$ can be visualized as follows:
player $\pI$ must fill, as
usual, a single row by playing a natural number at each of his
turn (thus he produces a real $x \in X$). Player $\pII$ is in charge of
filling again a table with 
$\omega$-many rows: she can pass, but at the end of the round she must
have enumerated a real $y_n$ on her $n$-th row (for each $n \in
\omega$).  The rules for  $\pII$
are that each 
$y_{\langle n,m \rangle }$ must 
code a set $D_{y,n,m} \in \bGamma$ whose intersection with $X$ is in
$\bDelta_\bGamma(X)$ (i.e.\ such that there is $P_{n,m} \in \breve{\bGamma}$
for which $P_{n,m} \cap X = D_{y,n,m} \cap X$),  and  for every
$n$ there must be an $m$
such that $x \in D_{y,n,m}$.  The output real $z$  is then defined by
$z(n) = m$ if and only if $m$ is the smallest 
 $k$ such that $x \in D_{y,n,k}$. 
As usual, any strategy $\tau$ for
$\bGamma$ can be seen as a product $\bigotimes_n \tau_n$ of legal
strategies for $\pII$ in $G_\W$, and one can define the projections
$\pi_n$ of strategies in $\sf{LS}_\bGamma$ in such a way that they ``commute'' with
the operation $\bigotimes_n$.

\begin{theorem}\label{theorGamma}
For every $X,A,B \subseteq \Bai$ and every $f \colon X \to \Bai$ we
have that:
\begin{enumerate}[i)]
\item $f \in \F_\bGamma$ if and only if there is some $\tau \in
  \sf{LS}_\bGamma$ such that $f = f_\tau$;
\item if $\pI$ has a winning strategy in $G_\bGamma(A,B)$, then
  $\pI$ has also a winning strategy in $G_\L(A,B)$.
\end{enumerate} 
\end{theorem}

\begin{proof}
First assume that $f \in \F_\bGamma$.
Since the sets $B_{n,m} = \{ z \in \Bai \mid z(n) = m \}$ form a
clopen subbasis for the usual topology of $\Bai$, we have that
$f^{-1}(B_{n,m}) \in
\bDelta_\bGamma(X)$. Let $S_{n,m} \in \bGamma$ be such that $S_{n,m}
\cap X = f^{-1}(B_{n,m})$,  $y_{n,m}$ be a
code for $S_{n,m}$, and  $x \in X$ be the real enumerated by
$\pI$: if we put $\tau = \bigotimes_n \tau_n$, where
$\tau_{\langle n,m \rangle} \in \sf{LS}_\W$  is any strategy
representing the constant function with value $y_{n,m}$, then $\tau$
is clearly a legal 
strategy for $\pII$ in $G_\bGamma$ such that $f = f_\tau$.

Assume now $\tau \in \sf{LS}_\bGamma$. Then 
\[ f_\tau^{-1} (B_{n,m}) = \{ x \in X \mid x
\in D_{x*\tau, n,m} 
\wedge \forall k < m (x \notin D_{x*\tau,n,k}) \}
\in \bGamma(X) \]
because  
$x \notin D_{x*\tau,n,k}\iff
x \notin P_{n,k}$, where $P_{n,k} \in
\breve{\bGamma}$ is such that $P_{n,k} \cap X = D_{x * \tau,n,k} \cap
X$. Hence $f_\tau \in
\F_\bGamma$. 

Finally, let $\rho$ be a winning strategy for $\pI$ in
$G_\bGamma(A,B)$, and $c_{\Bai}, c_\emptyset$ be codes for,
respectively, $\Bai$ and $\emptyset$ (as elements of $\bGamma$). Then the
strategy $\sigma$ for  $\pI$ in
$G_L(A,B)$ defined in the following way is clearly winning (the
proof being the same as in Theorem \ref{theorgames}):
\begin{quote}
Let $y \in \Bai$ be the real enumerated by $\pII$ in $G_\L(A,B)$,
and consider the run of the auxiliary game $G_\bGamma(A,B)$ in which $\pI$
plays according to $\rho$ and $\pII$  enumerates $c_{\Bai}$ or
$c_\emptyset$ on her $\langle n,m \rangle$-th row depending on
whether $y(n) = m$ or $y(n) \neq m$ (since $n \leq \langle n,m \rangle$
for any $m$, this strategy for $\pII$ is clearly legal). Then
copy at each turn the corresponding move made by $\pI$ in the run of
$G_\bGamma(A,B)$ described above.
\qedhere
\end{quote}
\end{proof}

Notice that, contrarily to the games defined in all the previous
sections, it is no more 
true that e.g.\ if $A,B \subseteq \bDelta^1_1$ then Borel
determinacy implies that $G_\bGamma(A,B)$, where $G_\bGamma = (\Bai,
M_\bGamma, R_\bGamma, \iota_\bGamma)$, is determined. 
This is because of the use of codes for sets in $\bDelta_\bGamma$,
which generally makes the set of rules more complicated than $\bGamma$
itself: in
fact, in most cases, to say that ``$x$ codes a
$\bDelta_\bGamma$-set'' 
require roughly speaking at least one real quantifier over a predicate
of the  same
complexity as $\bDelta_\bGamma$ (it is well-known e.g.\ that the set of codes
for the Borel sets  forms a  $\bPi^1_1$-complete  set).

Nevertheless, if $\bGamma \subsetneq \bDelta^1_1$ (that is if $\bGamma
= \bSigma^0_\xi$ for some countable $\xi$, being $\bGamma$ a
$\bSigma$-pointclass)  one can redefine the games
$G_{\bSigma^0_\xi}$ in such a way that the new sets of rules and the
interpretation functions remain
Borel (so  Borel determinacy will imply
that these new games are determined whenever $A,B \subseteq
\bDelta^1_1$).  This can be obtained by fixing in advance a sequence of
$\bPi^0_{\mu_n}$-complete sets $P_n$ (where $\seq{\mu_n}{n \in \omega}$ is an increasing
sequence of countable ordinals cofinal in $\xi$) as in the definition
of the games $G^\F_\xi$, and then using the fact
that for every $\bDelta^0_\xi(X)$ set $D \subseteq X$ (hence also for each
$f^{-1}(B_{n,m})$, where $f \in \F_{\bSigma^0_\xi}$ and the $B_{n,m}$'s
are defined as above) there is a $\bPi^0_{< \xi}(X)$-partition
$\seq{C_n}{n \in \omega}$ of $X$ such that $D = \bigcup_{i \in I} C_i$
for some $I \subseteq \omega$: roughly speaking, in the new games player
$\pII$ will have 
again to completely 
fill a board with $\omega$-many rows, but the function $f_\tau$ will
be determined by checking which of the reals that appears on the rows
of $\pII$'s table (instead of the real $x$ enumerated by $\pI$) 
belongs to the corresponding set $P_n$. We leave to the
reader the exact
definition of these games, as well as the proof that they represent
$\F_{\bSigma^0_\xi}$.\\

The same kind of construction introduced for the games $G_\bGamma$ allows
also to define  games for
$\sf{D}^{\vec{\F}}_\bGamma$ or $\tilde{\sf{D}}^{\vec{\F}}_\bGamma$
(where $\vec{\F}$ is a sequence of adequate playable sets 
of functions) for an arbitrary $\bSigma$-pointclass $\bGamma$. For
simplicity of presentation, we will deal again only with the case
$\sf{D}^\F_\bGamma$. The
idea is simply to take the game $G^\F_\xi$
and, instead of fixing 
in advance the control sets $P_n$, require $\pII$
to produce on each control row the \emph{code} for some control set in
$\bDelta_\bGamma(X)$:
a (non control) row will be activated just in
case the real $x \in X$ enumerated by $\pI$ belongs to the set  
$D \in \bDelta_\bGamma(X)$ coded
on the corresponding control row. More precisely, given a $\p$-closed
reduction game $G_* = (X, M_*, R_*, \iota_*)$ representing the
functions of $\F$ with domain $X$, define
$G^\F_\bGamma = (X, M^\F_\bGamma, R^\F_\bGamma, \iota^\F_\bGamma)$ as
follows (where $G_W = (\Bai, M_\W, R_\W, \iota_\W)$ is the Wadge game
representing continuous functions):

- $M^\F_\bGamma = M_* \cup  M_\W$;

- $R^\F_\bGamma = \{ (x,y) \in \Bai \times {}^\omega (\omega \cup
M^\F_\bGamma) \mid 
{\forall n [{(x,\pi_{2n}(y)) \in R_\W} \wedge
{D_{y,n} \cap X \in \bDelta_\bGamma(X)} \wedge
{(x,\pi_{2n+1}(y)) \in R_*}]} \wedge {\exists i (x \in
D_{y,i})} \}$, where $D_{y,n} = \{ x \in \Bai \mid
(\iota_\W(x,\pi_{2n}(y)),x) \in S \}$ and $S$ is
a fixed universal set for $\bGamma$; 

- $\iota^\F_\bGamma \colon X \to \Bai \colon (x,y) \mapsto
\iota_*(x,\pi_{2i+1}(y))$, where $i$ is smallest such that
$x \in 
D_{y,i}$.


Combining the proofs of Theorem \ref{theorgames} and Theorem \ref{theorGamma}
it is not hard to check that:

\begin{theorem}
For every $X,A,B \subseteq \Bai$ and every $f \colon X \to \Bai$ we
have that:
\begin{enumerate}[i)]
\item $f \in \sf{D}^\F_\bGamma$ if and only if there is some $\tau \in
  \sf{LS}^\F_\bGamma$ such that $f = f_\tau$;
\item if $\pI$ has a winning strategy in $G^\F_\bGamma(A,B)$, then
  $\pI$ has a winning strategy in $G_\L(A,B)$ as well.
\end{enumerate} 
\end{theorem}

\section{Determinacy and  applications to
  reducibilities}\label{sectionaxioms} 

In this section we will analyze the relationships among some
determinacy axioms, and show how to apply the techniques arising from
reduction games to the study of the reducibilities between
sets of reals induced by the corresponding sets of functions. Here we
will just
present two cases, namely the cases corresponding to $\Lip$ and
$\sf{D}^\W_\xi$ (for any fixed $\xi$). Notice that all reduction games used in
this section are always intended to be  of the form $G_* = (X, M_*,
R_*, \iota_*)$ \emph{with $X = \Bai$}.  

\subsection{$\Lip$-reducibilities}

We first consider the following  axioms which are related to the games
$G_{k\text{-}\Lip}$ ($k \in 
\omega$): this will lead in Theorem \ref{theoraxioms} to a slight
extension  of  the 
results concerning the equivalence of some determinacy axioms obtained by Andretta in his \cite{andrettaequivalence} and
\cite{andrettamoreonwadge}, although we must note that the most
difficult implication involved in such extension
was already proved in those papers.

\begin{description}
\item[$\AD(G_{k\text{-}\Lip})$] For every $A,B 
  \subseteq \Bai$ the game $G_{k\text{-}\Lip}(A,B)$ is determined.
\item[$\AD^\Lip_-$] For every
  $A,B \subseteq \Bai$ there is some $k \in \omega$ such that
  $G_{k\text{-}\Lip}(A,B)$ is determined.
\item[$\AD^\Lip$] For every
  $A,B \subseteq \Bai$ and for every $k \in \omega$ the game
  $G_{k\text{-}\Lip}(A,B)$ is 
  determined.
\end{description}

\begin{lemma}\label{lemmaaxioms}
\begin{enumerate}[i)]
\item $\AD \imp \AD^\Lip$;
\item $\AD^\Lip \iff \forall k \in \omega (\AD(G_{k\text{-}\Lip})) \imp \AD(G_{k\text{-}\Lip})
  \imp \AD^\Lip_- 
  \imp \SLO^\Lip \imp \SLOW$ for every $k \in \omega$;
\item $\ADL \imp \AD^\Lip$.
\end{enumerate}
\end{lemma}

\begin{proof}
 The first part is obvious since $\AD$ easily implies that every
  reduction game is 
determined, and the equivalence and the first two implications of part
\textit{ii)} are obvious as well. The third implication of part
\textit{ii)}  can be proved using the trivial observation that 
winning strategies for $\pI$ in any of the games $G_{k\text{-}\Lip}$ induce
contractions,  while the last implication 
follows from  $\Lip \subseteq 
\W$. 
It remains only to prove part \textit{iii)}. Fix $A,B
\subseteq \Bai$ and 
$k \in \omega$. If $B = \Bai$ then $G_{k\text{-}\Lip}(A,B)$ is trivially
determined ($\pII$ has a winning strategy if $A = \Bai$, and $\pI$ has a
winning strategy if $A \neq \Bai$): hence we can assume $B \neq \Bai$ and fix
some $y \notin B$.
Consider the auxiliary game $G=G_\L(A, 0^{(k)}\conc B)$: if $\pI$ has
a winning strategy 
$\sigma$ for $G$, $\pI$ can also win $G_{k\text{-}\Lip}(A,B)$
simply playing $\sigma(0^{(i)})$ for the
first $k$ turns (i.e.\ for $i \leq k$), and then playing
$\sigma(0^{(k)} \conc s)$ if $\pII$ has enumerated a sequence of the
form $\p^{(k)} \conc s$ (for some $s \in {}^{< \omega} \omega$) in
the game $G_{k\text{-}\Lip}(A,B)$, and $0$ otherwise. Conversely, if $\pII$ 
has a winning 
strategy $\tau$ in the game $G$, then
she can also win $G_{k\text{-}\Lip}(A,B)$ by playing $\p$ for the first $k$
rounds, and then 
playing $\tau(s)$, where $s$ is the sequence enumerated by $\pI$ in
$G_{k\text{-}\Lip}(A,B)$, 
if $(\restr{s}{k}) * \tau = 0^{(k)}$, or enumerating $y$ otherwise.
\end{proof}

\begin{theorem}[\BP+\DCR] \label{theoraxioms}
Let $\ax{Ax}$ be one of the axioms $\AD(G_{k\text{-}\Lip})$,
$\AD^\Lip_-$ and $\AD^\Lip$. Then $\ax{Ax}
\iff \SLOW$.
\end{theorem}

\begin{proof}
  By Lemma \ref{lemmaaxioms} we have $\ADL \imp \ax{Ax} \imp \SLOW$, and since
  under $\BP +\DCR$ we have from \cite[Proposition 15 and Theorem 18]{andrettamoreonwadge}
  that   $\SLOW \imp 
\ADL$, we get the desired equivalence.
\end{proof}

In particular, this theorem implies that
all the results about the $\Lip$-hierarchy obtained in
\cite{mottorosbairereductions} (such as the fact that the structure
induced by $\leq_\Lip$ can be completely determined and is a
well-founded semi-linear order whose antichains have size at most 
two, or the relationship between this hierarchy and the ones induced
by $\leq_\L$ and $\leq_\W$)  hold under any
of the axioms listed above (together with $\BP+\DCR$).

\subsection{$\sf{D}^\W_\xi$-reducibilities}

We now turn our attention to the $\sf{D}^\W_\xi$-hierarchies (for some
fixed nonzero $\xi < \omega_1$). Recall from 
\cite[pp.\ 47-48]{mottorosborelamenability} that there are two operations
$\Sigma^\xi$ and $\Pi^\xi$ such that $\{\Sigma^\xi(A),\Pi^\xi(A)\}$
are the successors of $A$ in the $\sf{D}^\W_\xi$-hierarchy whenever $A
\leq_{\sf{D}^\W_\xi} \neg A$. Here are the definitions: 
\[ \Sigma^\xi(A) = \{x \in \Bai \mid \exists n (\pi_{2n}(x) \in P_n
\wedge \forall i<n(\pi_{2i}(x) \notin P_i) \wedge \pi_{2n+1}(x) \in
A)\}\]
and
\[ \Pi^\xi(A) = \Sigma^\xi(A) \cup R_\xi,\]
where the $P_n$'s are $\bPi^0_{\mu_n}$-complete for an increasing
sequence of ordinals $\seq{\mu_n}{n \in \omega}$ cofinal in $\xi$ and $R_\xi =
\{x\in \Bai \mid \forall n(\pi_{2n}(x) \notin P_n) \}$.
There is a strict relationship between the games $G^\F_\xi$ (in
particular when $\F = \W$) and these
successor operations --- in fact our definition of $G^\F_\xi$ was 
originally motivated by the definitions of $\Sigma^\xi$ and $\Pi^\xi$.

\begin{proposition}\label{propgamessuccessor}
For every $A,B \subseteq \Bai$, the following are equivalent:
\begin{enumerate}[i)]
\item $A \leq_{\sf{D}^\W_\xi} B$;
\item $A \leq_\W \Sigma^\xi(B)$ \emph{and} $A \leq_\W \Pi^\xi(B)$;
\item $A \leq_\W \Sigma^\xi(B)$ via some function $f$ such that ${\rm range}
(f)\cap R_\xi = \emptyset$;
\item $A \leq_\W \Pi^\xi(B)$ via some function $f$ such that ${\rm range}
(f)\cap R_\xi = \emptyset$.
\end{enumerate}
\end{proposition}

\begin{proof}
  Obviously, \textit{iii)}$\iff$\textit{iv)} since $\Sigma^\xi(B)
  \setminus R_\xi
= \Pi^\xi(B) \setminus R_\xi$ for every set $B \subseteq \Bai$. Moreover,
\textit{iii)} and \textit{iv)} together trivially imply \textit{ii)}, and
\textit{i)} implies \textit{iii)} and \textit{iv)} since, by definition
of $G^\W_\xi$, every winning  strategy for $\pII$ in $G^\W_\xi(A,B)$
can be obviously converted into
a winning strategy for $\pII$ in both $G_\W(A,\Sigma^\xi(B))$ and
$G_\W(A,\Pi^\xi(B))$. To see that
\textit{ii)} implies \textit{i)},
let $\sigma^0$ and $\sigma^1$ be, respectively, winning strategies
for $\pII$ in $G_\W(A,\Sigma^\xi(B))$ and $G_\W(A,\Pi^\xi(B))$. As already
observed in Claim \ref{claimchangecontrolsets}, we can change the sets
$P_n$ in the definition of $G^\W_\xi$ with 
some suitable $\hat{P}_n$'s,  and it will suffice to show that $\pII$ has a
winning strategy
in $\hat{G}^\W_\xi(A,B)$, where $\hat{G}^\W_\xi$ is the game defined using
the  $\hat{P}_n$'s
instead of the $P_n$'s.
Choose for every $n \in \omega$ and $i=0,1$ a strategy $\sigma^i_n
\in \sf{LS}_\W$ representing $\pi_n \circ f_{\sigma^i}$. Then put
$\hat{P}_{2n} = \hat{P}_{2n+1} = P_n$ for 
every $n \in \omega$ and set $\tau = \bigotimes_n \tau_n$, where
$\tau_{4k} =  \sigma^0_{2k}$,   
$\tau_{4k+1}=  \sigma^0_{2k+1}$,
$\tau_{4k+2} =  \sigma^1_{2k}$ and 
$\tau_{4k+3} =  \sigma^1_{2k+1}$ ($k \in \omega$).
Notice that each $\hat{P}_n$ is $\bPi^0_{\hat{\mu}_n}$-complete, where
$\hat{\mu}_{2k+i} = \mu_k$ for $k \in \omega, i = 0,1$ (so that
$\seq{\hat{\mu}_n}{n \in \omega}$ is a sequence of ordinals cofinal in
$\xi$).
We claim that $\tau \in \hat{\sf{LS}}^\W_\xi$. Let $x \in \Bai$. 
 Since the $\sigma^i$'s are \emph{winning} strategies
in the corresponding games, we have that for every real $x$
\[ f_{\sigma^0}(x) \in R_\xi \imp x \notin A \imp f_{\sigma^1}(x)
\notin R_\xi,\] 
thus there must be some $n$ such that either $\pi_{2n} (
f_{\sigma^0}(x)) \in P_n$ 
or else $\pi_{2n} (f_{\sigma^1}(x)) \in P_n$. But this implies
that either 
$\iota_\W(x,\pi_{4n}(x*\tau)) \in \hat{P}_{2n}$ or
$\iota_\W(x,\pi_{4n+2}(x*\tau)) \in 
\hat{P}_{2n+1}$, hence 
$\tau$ is 
legal.  
To finish the proof, let $n$ be
the smallest natural number such that $\iota_\W(x,\pi_{2n}(x * \tau))
\in \hat{P}_n$:  then if
$n=2k+i$ ($k \in \omega, i=0,1$) we clearly have
$\iota_\W(x,\pi_{2n+1}(x * \tau)) = \pi_{2k+1} (f_{\sigma^i}(x))$ and
\[ \pi_{2k+1} (f_{\sigma^i}(x))
\in B \iff f_{\sigma^i}(x) \in B_i \iff x \in A,\]
 where $B_0 = \Sigma^\xi(B)$ and $B_1=\Pi^\xi
(B)$. Therefore 
$\tau$ is a winning strategy for $\pII$ in $\hat{G}^\W_\xi(A,B)$.
\end{proof}

As for the Lipschitz games $G_{k\text{-} \Lip}$, the games $G^\W_\xi$
allow us to introduce new determinacy axioms (one for each $\xi$):
\begin{description}
\item[$\AD^\W_\xi$] For every $A,B \subseteq \Bai$ the game $G^\W_\xi(A,B)$ is
determined.
\end{description}

Clearly $\AD \imp \AD^\W_\xi \imp \SLO^{\sf{D}^\W_\xi}$,
but Proposition
\ref{propgamessuccessor} allows us to prove the following stronger
corollary.

\begin{corollary}\label{cordeterminacy}
  $\ADW$ implies $\AD^\W_\xi$.
\end{corollary}

\begin{proof}
  Let $A$ and $B$ be two subsets of $\Bai$. By $\ADW$, the games
$G_\W(A,\Sigma^\xi(B))$ and $G_\W(A,\Pi^\xi(B))$ are determined. If
$\pI$ has a winning strategy in one of these two games, then he can
convert this
strategy into a winning strategy for $\pI$ in $G^\W_\xi(A,B)$ in  the obvious
way, hence we can assume that 
$\pII$ wins both the games. But in this case $\pII$ has a winning
strategy in $G^\W_\xi(A,B)$ by Proposition \ref{propgamessuccessor}, hence
we are done.
\end{proof}

There is a natural question arising from the previous corollary, namely:

\begin{question}\label{questionconverse}
 Assume $\BP + \DCR$. Given a  countable ordinal $\xi>1$, does the
 converse to Corollary \ref{cordeterminacy} hold? 
\end{question}

This question was answered positively for $\xi =2$ by Andretta in his
\cite{andrettamoreonwadge}, where it is shown that in fact
$\SLO^{\sf{D}_2}$ (which is a direct consequence of $\AD^\W_2$)
implies $\AD^\W$ if we assume $\BP+\DCR$. The proof is carried out
with an induction on the $\sf{D}_2$-hierarchy of degrees (which can be
determined under $\SLO^{\sf{D}_2} + \BP + \DCR$), but even if we will
show that for any $\xi$  the axioms $\AD^\W_\xi + \BP + \DCR$
are indeed strong enough to determine the $\sf{D}_\xi$-hierarchy of
degrees as well (see Theorem \ref{theorstruc} below), it seems that
the argument used by Andretta does not generalize in a straightforward
way to higher levels. Therefore Question \ref{questionconverse} is
still completely open for $\xi \geq 3$. 

We will now prove that, as announced in the previous paragraph, the
axiom $\AD^\W_\xi$ is strong 
enough\footnote{All the following results can also be proved assuming
  that for every $A,B \subseteq \Bai$ either $A \leq_{\sf{D}^\W_\xi} B$
  or $\neg B \leq_{\sf{c}} A$, which by Theorem
  \ref{theorgames} is a (seemingly weaker) consequence of
  $\AD^\W_\xi$.} to determine (together with  
$\BP$ and $\DCR$) the degree-structure
induced by  $\sf{D}^\W_\xi$, or even by any  Borel-amenable set of reductions
$\F \supseteq \sf{D}^\W_\xi$ (see \cite{mottorosborelamenability} for
a general introduction to such degree-structures). This shows that to study
a Borel-amenable reducibility $\F$ we just need to assume
an axiom which is ``of the same level'' of $\F$, rather than the seemingly
stronger $\SLO^\W$. 

Toward our goal, we will
simply 
modify the arguments presented in \cite{mottorosborelamenability}
whenever an axiom stronger than $\AD^\W_\xi$ was required. We start by
proving a lemma 
(analogous to  \cite[Lemma 2.1]{mottorosborelamenability}) under the
new axiomatization. 

\begin{lemma}\label{lemmabasic2}
  Assume $\AD^\W_\xi$. For every set of reductions $\F \supseteq
\sf{D}^\W_\xi$ and every $A,B \subseteq \Bai$ we have $A <_\F B \imp A<_\L B$.
\end{lemma}

\begin{proof}
  Since $\AD^\W_\xi \imp \SLO^{\sf{D}^\W_\xi} \imp \SLO^\F$, from $A <_\F B$ we
  have $A <_\F \neg B$. 
But then $\pII$ cannot win $G^\W_\xi(\neg B,A)$ (if this would happen, then
$\neg B \leq_{\sf{D}^\W_\xi} A$ and hence also $\neg B \leq_\F A$).
 Therefore by $\AD^\W_\xi$ we have that
$\pI$ has a winning strategy in the same game, and hence $A \leq_\L B$
by Theorem 
\ref{theorgames}. Moreover $B \nleq_\L A$ since otherwise $B \leq_\F A$
(here we use the fact that $\L \subseteq \sf{D}^\W_\xi \subseteq \F$), and thus
$A <_\L B$.
\end{proof}

\begin{lemma}\label{lemmawellfounded}
  Assume $\AD^\W_\xi+\BP+\DCR$. For every set of reductions $\F
\supseteq \sf{D}^\W_\xi$ the relation $\leq_\F$ is well-founded.
\end{lemma}

\begin{proof}
  It clearly suffices to prove that there is no $\leq_\F$ descending
  chain --- the equivalence between this statement and well-foundness can
  be obtained in the usual way using the existence of a surjection $j
  \colon \Bai \onto \F$ (see 
  \cite[Corollary 2.2]{andrettaslo}). 
So
assume towards a contradiction that $A_0 >_\F A_1 >_\F \dotsc$ is such
a chain. We claim that $A_{n+1} \leq_{\sf{c}} A_n$ and $A_{n+1}
\leq_{\sf{c}} \neg A_n$ for every $n \in \omega$, i.e.\ that $\pI$ wins
both $G_\L(A_n,A_{n+1})$ and $G_\L(\neg A_n,A_{n+1})$: applying
then the classic Martin-Monk argument to these winning strategies, we
can construct the 
flip-set which contradicts $\BP$, finishing our proof.
First note that for every $n \in \omega$, we have
$A_{n+1} <_\F \neg A_n$ by $\SLO^\F$ (which follows from $\AD^\W_\xi$). 
From this fact one can conclude, arguing as in Lemma \ref{lemmabasic2}, 
that $\pII$ cannot win neither $G^\W_\xi(A_n, A_{n+1})$ nor
$G^\W_\xi(\neg A_n,A_{n+1})$: but then $\pI$ wins both games by $\AD^\W_\xi$,
and hence $A_{n+1} \leq_{\sf{c}} A_n,\neg A_n$ by Theorem \ref{theorgames}.
\end{proof}


\begin{theorem}\label{theorstruc}
  Assume $\AD^\W_\xi+\BP+\DCR$, and let $\F \supseteq \sf{D}^\W_\xi$
  be any Borel-amenable set of 
reductions. Then the degree-structure
induced by $\leq_\F$ is completely determined and looks like the
Wadge one.
\end{theorem}

\begin{proof}
  Since $\AD^\W_\xi \imp \SLO^\F$ and $\leq_\F$ is well-founded by Lemma
\ref{lemmawellfounded}, we have that Theorem  3.1 
and Theorem 4.6
of
\cite{mottorosborelamenability} are 
provable under our new axiomatization (for part
\textit{vii)} of Theorem 3.1, assume $B \equiv_\F A \nleq_\F
\neg A$: since $\pII$ cannot win neither $G^\W_\xi(B,\neg A)$ nor
$G^\W_\xi(\neg A,B)$, we have $A
\leq_\L B$ and $B \leq_\L A$ by Theorem \ref{theorgames}). 
Moreover,  \cite[Theorem 5.3]{mottorosborelamenability} follows from
$\BP$ alone, hence we are done.
\end{proof}

Finally, notice that we can also reprove Theorem 4.7 of
\cite{mottorosborelamenability} 
in this new 
context using Lemma \ref{lemmabasic2} instead of 
\cite[Lemma 2.1]{mottorosborelamenability}: therefore under
$\AD^\W_\xi+\BP+\DCR$ we have that for every pair $\F,\G \supseteq
\sf{D}^\W_\xi$ of Borel-amenable sets of reductions, $\F$ is
equivalent to  
(i.e.\ induces the same hierarchy of degrees as) $\G$ just in case they
have the same characteristic set, that is just in case
\[ \Delta_\F = \{ D \subseteq \Bai \mid D \leq_\F \bN_{\langle 0 \rangle}\} =
\{ D \subseteq \Bai \mid D \leq_\G \bN_{\langle 0 \rangle}\}  = \Delta_\G. \]

\section{Non adequate playable set of functions}\label{sectionsmallness} 

This final section is devoted to a technical refinement of the notion
of being adequate for a 
certain
set of functions in relationship to the possibility of representing 
such set by means of
reduction games (using the ideas coming from Sections \ref{sectionplayable} and
\ref{sectionGamma}). 

We start with a significative example. As observed after Definition
\ref{defadequate}, 
the class of all Lipschitz functions $\Lip$ is not adequate (being not
$\p$-closed), but still
it is possible 
(and useful) to consider e.g.\ classes of the form 
$\sf{D}^\Lip_\xi$, which are proper subsets of $\sf{D}^\W_\xi$ (see
\cite[pp.\ 45-46]{mottorosborelamenability}). 
Roughly speaking, in order to define the game $G^\Lip_\xi$ (whose
legal strategies 
for $\pII$ will induce the functions in $\sf{D}^\Lip_\xi$), it is enough to
modify
the algorithm\footnote{Just forbidding $\pII$ to pass (that is 
  making $\pII$  always play a natural number on some of her rows)
  does not give 
  the desired result, because legal strategies for $\pII$ in such
   game would induce functions uniformly continuous (rather than
   Lipschitz) on a definable 
  partition and these two sets of functions are distinct by \cite[pp.\
  45-46]{mottorosborelamenability}.} that $\pII$ must follow to fill
in the $\omega$-many rows 
of her table in the game $G^\F_\xi$: in fact, in the game $G^\Lip_\xi$, 
$\pII$ will have to \emph{simultaneously} play  a
new natural number on 
a certain \emph{finite set} of rows at each of her turns. More
precisely: for $0  \neq k  \in \omega$ let
$\seq{s^k_n}{n \in \omega}$ be 
an enumeration without repetitions of ${}^k \omega$, and let
$\seq{\mu_n}{n \in \omega}$ and $\seq{P_n}{n \in \omega}$ be chosen
as in Subsection \ref{sectionpiecewise}. Then define $G^\Lip_\xi = (X,
M^\Lip_\xi, R^\Lip_\xi, 
\iota^\Lip_\xi)$ by: 

- $M^\Lip_\xi = \emptyset$;

- for $y \in \Bai$, $k \in \omega$ and $i = 0,1$ define $y_{2k+i} =
\seq{s^{2n+2}_{y(n)}(2k+i)}{n \geq k}$;

- $R^\Lip_\xi = \{ (x,y) \in \Bai \times \Bai \mid \exists n
(y_{2n} \in P_n) \}$;

- $\iota^\Lip_\xi \colon R^\Lip_\xi \to \Bai \colon (x,y) \mapsto
y_{2n+1}$ where $n$ is smallest such that $y_{2n} \in P_n$.

As for the games $G^\F_\xi$, a strategy $\tau = \bigotimes'_n \tau_n$
for $\pII$  in
$G^\Lip_\xi$ can be constructed
from a sequence $\langle \tau_n\mid n \in \omega \rangle$ of
strategies for $\pII$ in $G_{k_n\text{-}\Lip}$, where $k_n$ is the
unique natural number such that either $n=2k_n$ or $n=2k_n+1$: in fact
it is enough to define $\tau(s) = m$, where $m$ is such that
$\seq{\tau_n(s)}{n < 2 \leng(s)} = s^{2 \leng(s)}_m$, and it is easy to
check that $\tau \in \sf{LS}^\Lip_\xi$ 
just in case for every $x \in \Bai$ there is some $n \in \omega$
such that $\iota_{k_n\text{-}\Lip}(x * \tau_{2n}) \in P_n$.
Conversely, given a strategy $\tau \in \sf{LS}^\Lip_\xi$ one can 
construct the strategies
$\pi'_{2k+i}(\tau) \in \sf{LS}_{k\text{-}\Lip}$ (for $i=0,1$) by
letting $\pi'_{2k+1}(\tau)(s) = \p$ is $\leng(s) < k$ and
$\pi'_{2k+1}(\tau)(s) = 
s^{2 \leng(s)}_{\tau(s)}(2k+i)$ otherwise.

\begin{theorem}\label{theorGLipxi}
 For every $X,A,B \subseteq \Bai$ and every $f \colon X \to \Bai$ we
 have that:
\begin{enumerate}[i)]
\item $f \in \sf{D}^\Lip_\xi$ if
and only if there is some $\tau \in \sf{LS}^\Lip_\xi$
such that $f=f_\tau$;

\item if $\pI$ has a winning strategy in $G^\Lip_\xi(A,B)$, then $\pI$
  has also a winning strategy in $G_\L(A,B)$. 
\end{enumerate}
\end{theorem}

\begin{proof}
  Assume first that $f \in \sf{D}^\Lip_\xi$, and let $\{f_k\mid k \in
  \omega \} \subseteq \Lip$ and $\langle D_k\mid
  k \in \omega \rangle$ be a sequence of $\bPi^0_{<\xi}$-sets such
  that $\seq{D_k \cap X}{k \in \omega}$ is a partition of $X$ such that
  $\restr{f}{D_k}=\restr{f_k}{D_k}$. By Borel
  determinacy\footnote{Using Borel determinacy, if
    $n_k$ is such that $D_k \in \bPi^0_{\mu_{n_k}}$ then $D_k \leq_\L
    P_{n_k}$: otherwise $P_{n_k} \leq_\L \neg D_k \in
    \bSigma^0_{\mu_{n_k}}$ by $\SLO^\L$ for Borel sets, contradicting
    the $\bPi^0_{\mu_{n_k}}$-properness of $P_{n_k}$.},  we can find
  an increasing sequence $\langle n_k\mid  k \in 
  \omega \rangle$ of natural numbers such that $D_k \leq_\L P_{n_k}$.
Moreover, let $i_k$ be smallest such that $f_k \in
\Lip(2^{i_k})$, and inductively define $m_0=\max\{n_0,i_0\}$ and $m_{k+1}
= \max\{n_{k+1},i_{k+1},m_k+1\}$, so that $D_k \leq_\L P_{m_k}$ (hence also $D_k \leq_{\Lip(2^{m_k})} P_{m_k}$) and $f_k
\in \Lip(2^{m_k})$ for every $k \in \omega$. Let $\sigma_k$ be a winning
strategy for $\pII$ in $G_{m_k\text{-}\Lip}(D_k,P_{m_k})$ and
$\hat{\tau}_k \in \sf{LS}_{m_k\text{-}\Lip}$ be such that $f_k =
f_{\hat{\tau}_k}$. Finally, fix 
$y_n \notin P_n$ and for every $y \in \Bai$ let $\rho^k_y \in
\sf{LS}_{k\text{-}\Lip}$  be such that $f_{\rho^k_y}$ is constantly
equal to 
$y$.
Now define $\tau_{2n} = \sigma_k$ and $\tau_{2n+1} = \hat{\tau}_k$ if
$n=m_k$,
and $\tau_{2n} = \rho^n_{y_n} = \tau_{2n+1}$ otherwise. If we construct
the strategy $\tau = \bigotimes'_n \tau_n$ for $\pII$ in $G^\Lip_\xi$ as explained above, it
is not hard to check that $x \in D_k$ if and only if 
$(x * \tau)_{2m_k} \in P_{m_k}$, and that in this case $(x*\tau)_{2n} \notin P_n$ for
$n \neq m_k$ and
\[ f_\tau(x) = (x*\tau)_{2m_k+1} = \iota_{m_k\text{-}\Lip}(x *
\hat{\tau}_k) = f_k(x) = f(x).\]

Conversely, let $F_n$ be the set of those $x$ for which $n$ is least
such that $(x * \tau)_{2n}\in 
P_n$. Clearly these $F_n$'s form a $\bDelta^0_\xi$-partition of
$X$ and, as already observed, $\pi'_{2n+1}(\tau)$ induces a function $f_n
\in \Lip(2^n)$. Thus $f_\tau = \bigcup_{n \in \omega}(\restr{f_n}{F_n})$
and we are done.

Finally, the second part of the theorem can be proved as in Theorem
\ref{theorgames} (although the coding of the strategies involved is 
more complicated).
\end{proof}

Theorem \ref{theorGLipxi} clearly allows us to reprove all the
results of Section 
\ref{sectionaxioms} 
(except for Proposition
\ref{propgamessuccessor} and its corollary),  but using
$G^\Lip_\xi$ and $\sf{D}^\Lip_\xi$ instead of $G^\W_\xi$ and
$\sf{D}^\W_\xi$. The obstacle in reproving also  Proposition
\ref{propgamessuccessor}
  simply relies in the definitions of the operations $\Sigma^\xi$ and
  $\Pi^\xi$: nevertheless, using the ``coding'' introduced in the definition of $G^\Lip_\xi$ it is possible to define a set
  $\hat{R}_\xi$ and two new operations 
  $\hat{\Sigma}^\xi$ and $\hat{\Pi}^\xi$ such that $\Sigma^\xi(A)
  \equiv_\W \hat{\Sigma}^\xi(A)$ and $\Pi^\xi(A) \equiv_\W
  \hat{\Pi}^\xi(A)$ for every $A \subseteq \Bai$, and with the further
  property that 
  Proposition \ref{propgamessuccessor} holds whenever we replace all 
  occurrences of $\sf{D}^\W_\xi$, $\Sigma^\xi$, $\Pi^\xi$ and
  $R_\xi$ in its statement with $\sf{D}^\Lip_\xi$, $\hat{\Sigma}^\xi$,
  $\hat{\Pi}^\xi$ and $\hat{R}_\xi$.

From the construction of $G^\Lip_\xi$ we can now infer which are
the minimal
conditions on the $\F_n$'s under which one can carry out the
constructions above and define the games $G^{\vec{\F}}_\xi$ which represent
 $\sf{D}^{\vec{\F}}_\xi$.

\begin{defin} \label{defdelayable}
Let $G_* = (X,M_*,R_*,\iota_*)$ be a reduction game and $\F_*$ be the
set of functions  
induced by legal strategies for $\pII$ in $G_*$.
We say that $G_*$
(or $\F_*$) is \emph{delayable}  
if for every $n \in \omega$ the set $\F_*$ is still represented by each of the
new reduction  
games $G^n_* = (X, M^n_*, R^n_*, \iota^n_*)$ defined by:

- $M^n_* = M_* \cup \{ \p \}$, where $\p$ is a new symbol not in $M_*$;

- $R^n_* = \{ (x,y) \in \Bai \times {}^\omega (\omega \cup M^n_*) \mid
{\forall k (y(k) = \p \iff k < n)} \wedge {(x, \seq{y(n+k)}{k \in
    \omega}) \in R_*})$;

- $\iota^n_* \colon R^n_* \to \Bai \colon (x,y) \mapsto
\iota_*(x,\seq{y(n+k)}{k 
  \in \omega})$.
\end{defin}

As for $\p$-closure, one could note that the property of being
delayable corresponds to the property of being closed under
right-composition with Lipschitz functions from $X$ into
itself.  
From this definition and from the construction above, it turns out
that we can still define reduction games
representing 
$\sf{D}^{\vec{\F}}_\xi$, $\tilde{\sf{D}}^{\vec{\F}}_\xi$ and $\lim
\vec{\F}$ whenever each element $\F_n$ of the sequence $\vec{\F}$
contains the identity function and is represented by a parametrized class
$\G_*$ of \emph{delayable} (rather than $\p$-closed) reduction games.

This technical condition is optimal if we want to define reduction
games like those presented in Section \ref{sectionplayable}, in
which $\pII$ has to fill in a table with $\omega$-many rows: in fact,
any reduction game is by definition formalizable as a game on
$\omega$, and this essentially means that in each turn $\pII$ can make
at most a 
finite numbers of moves on a finite number of rows of her table,
condition which easily leads to our definition of delayability. However,
there are still examples of natural playable sets of functions (and
even of reducibilities for sets of reals, like the set $\L$) which are
clearly non-delayable. This leaves open the following question:

\begin{question}
Given $\xi>1$ and a non-delayable playable set of functions $\F$, are
there reduction games representing the  classes of 
functions  $\sf{D}^\F_\xi$ and $\tilde{\sf{D}}^\F_\xi$? In particular,
are there reduction games representing $\sf{D}^\L_\xi$ and
$\tilde{\sf{D}}^\L_\xi$? 
\end{question}

\begin{acknowledgement}
 Research partially supported by FWF (Austrian Research
 Fund) through 
  Project number P 19898-N18.
\end{acknowledgement}


\providecommand{\WileyBibTextsc}{}
\let\textsc\WileyBibTextsc
\providecommand{\othercit}{}
\providecommand{\jr}[1]{#1}
\providecommand{\etal}{~et~al.}

\end{document}